\documentclass[a4paper,12pt]{article}
\usepackage[utf8]{inputenc}
\usepackage[english]{babel}
\usepackage{hyperref}
\usepackage{amssymb,amsmath,amsthm}
\usepackage{amsthm}
\usepackage{enumerate}
\usepackage{enumitem}
\usepackage{amssymb}
\usepackage{graphicx, color}
\usepackage{epsfig}
\usepackage{tikz}
\usetikzlibrary{calc,decorations.pathmorphing,decorations.text}
\usepackage{subcaption}
\usepackage{refcount}
\usepackage[hmargin=2.5cm,vmargin=3cm]{geometry}
\usepackage{mathtools, braket}
\usepackage{authblk}
\usepackage{centernot}
\usepackage{indentfirst}
\usepackage{float}
\theoremstyle{plain}
\newtheorem{thm}{Thm}[section]
\newtheorem{question}[thm]{Question}

\newtheorem{claim}[thm]{Claim}
\newtheorem{theorem}[thm]{Theorem}

\newtheorem{definition}[thm]{Definition}

\newtheorem{innercustomgeneric}{\customgenericname}
\providecommand{\customgenericname}{}
\newcommand{\newcustomtheorem}[2]{%
	\newenvironment{#1}[1]
	{%
		\renewcommand\customgenericname{#2}%
		\renewcommand\theinnercustomgeneric{##1}%
		\innercustomgeneric
	}
	{\endinnercustomgeneric}
}

\newcustomtheorem{customtheorem}{Theorem}
\newcustomtheorem{customlemma}{Lemma}

\newenvironment{proof*}{\noindent \emph{Proof.}}{\hfill$\Diamond$}

\renewcommand{\pod}[1]{\allowbreak\mathchoice
	{\if@display \mkern 0mu\else \mkern 0mu\fi (#1)}
	{\if@display \mkern 0mu\else \mkern 0mu\fi (#1)}
	{\mkern 1mu(\mathrm{mod}\mkern 4mu #1)}
	{\mkern 0mu(#1)}
}

\usepackage[color=green!30]{todonotes}

\usepackage{caption}
\usetikzlibrary{matrix}
\usetikzlibrary{decorations.markings}
\tikzstyle{vertex}=[circle, draw, fill=black!50,
inner sep=0pt, minimum width=4pt]
\tikzset{->-/.style={decoration={
			markings,
			mark=at position .5 with {\arrow{>}}},postaction={decorate}}}

\tikzstyle{bigblue}=[color=blue, very thick, >=stealth]
\tikzstyle{lightblue}=[color=blue, thin, >=stealth]

\tikzstyle{bigred}=[color=red, very thick, >=stealth]
\tikzstyle{lightred}=[color=red, thin, >=stealth]

\tikzstyle{biggreen}=[color=black!30!green, very thick, >=stealth]
\tikzstyle{lightgreen}=[color=black!30!green,  thin, >=stealth]

\title{Linear saturation numbers of Berge-$C_3$ and Berge-$C_4$}

\author{\small Changxin Wang$^1$,  Junxue Zhang$^1$\footnote{Corresponding author} \\ {\small$^1$  Center for Combinatorics and LPMC, Nankai University, Tianjin 300071, China}\\
{\small Emails: Simonang@163.com; jxuezhang@163.com}}
\date{}

\begin{document}\maketitle
\def\aky{empty}
\pagestyle{empty}






\begin{abstract}
The linear saturation number $sat^{lin}_k(n,\mathcal{F})$ (linear extremal number $ex^{lin}_k(n,\mathcal{F})$) of $\mathcal{F}$ is the minimum (maximum) number of hyperedges of
an $n$-vertex linear $k$-uniform hypergraph containing no member of $\mathcal{F}$ as a subgraph, but the addition of  any new hyperedge such that the result hypergraph is still a linear $k$-uniform hypergraph creates a copy of some hypergraph in $\mathcal{F}$.  
Determining $ex_3^{lin}(n$, Berge-$C_3$) is  equivalent to the famous (6,3)-problem,  which has been settled in 1976. Since then, 
determining the linear extremal numbers of Berge cycles was extensively studied. 
As the counterpart of this problem in saturation problems, the problem of determining  the linear saturation numbers of   Berge cycles is considered.  In this paper,
we prove that $sat^{lin}_k$($n$, Berge-$C_t)\ge \big\lfloor\frac{n-1}{k-1}\big\rfloor$ for any integers $k\ge3$, $t\ge 3$, and the equality holds if $t=3$. 
In addition, we provide 
an upper bound for $sat^{lin}_3(n,$ Berge-$C_4)$ and 
for any disconnected Berge-$C_4$-saturated linear 3-uniform  hypergraph, we give a lower bound for the number of hyperedges  of it. 
\\[2mm]

\end{abstract}

\section{Introduction}

A \emph{hypergraph} $H=(V(H),E(H))$ consists of a vertex set $V(H)$ and a hyperedge set $E(H)$, where each hyperedge $e\in E(H)$ is a subset of $V(H)$. 
A hypergraph is called \emph{$k$-uniform} if each member of $E(H)$ has size exactly $k$ and called \emph{linear} if any two different hyperedges have at most one vertex in common. 

Given a family of $k$-uniform hypergraphs $\mathcal{F}$ and a $k$-uniform hypergraph $H$, we say that $H$ is 
\emph{$\mathcal{F}$-saturated} if $H$ dose not contain any member of $\mathcal{F}$ as a subgraph but the addition of any new $k$-edge creates a copy of some hypergraph in $\mathcal{F}$.
The \emph{extremal number} of $\mathcal{F}$, denoted $ex_k(n,\mathcal{F})$, is defined as 
$$ex_k(n,\mathcal{F})=\max\{|E(H)|: |V(H)|=n \text{ and  H is $\mathcal{F}$-saturated}\},$$
and the 
\emph{saturation number} of $\mathcal{F}$, denoted $sat_k(n,\mathcal{F})$, is defined as 
$$sat_k(n,\mathcal{F})=\min\{|E(H)|: |V(H)|=n \text{ and  H is $\mathcal{F}$-saturated}\}.$$
Given a set $V$ and a positive integer $k$, $\binom{V}{k}$ denotes the set of subsets of $V$ with size exactly $k$. 
We say that $H$ is \emph{linear $\mathcal{F}$-saturated} if $H$ is linear and $H$ dose not contain any member of $\mathcal{F}$ as a subgraph but adding any $e\in \binom{V(H)}{k} \setminus E(H)$ to $H$ such that  $H+e$ is still linear creates a  copy of some hypergraph in $\mathcal{F}$. In this paper, we   confine ourselves to $k$-uniform linear hypergraphs. We only consider to add the hyperedges that can  make the hypergraph  preserve linear $k$-uniform.    
Let $$ex_k^{lin}(n,\mathcal{F})=\max\{|E(H)|: |V(H)|=n \text{ and  H is linear $\mathcal{F}$-saturated}\},$$ and 
$$sat_k^{lin}(n,\mathcal{F})=\min\{|E(H)|: |V(H)|=n \text{ and  H is linear $\mathcal{F}$-saturated}\}.$$
We shall refer to $ex_k^{lin}(n,\mathcal{F})$ as the \emph{linear extremal number} of $\mathcal{F}$ and 
$sat_k^{lin}(n,\mathcal{F})$  as the 
\emph{linear saturation number} of $\mathcal{F}$.


The Berge hypergraph  is one of the most popular kind of hypergraphs in the saturation problems of hypergraphs.  
The classical notions of  a cycle and a path  in  hypergraphs are a Berge cycle and a Berge path. A \emph{Berge cycle} Berge-$C_\ell$ of length $\ell$ is an alternating sequence of distinct vertices and hyperedges of the form $(v_1,e_1,v_2,e_2,\dots ,v_\ell,e_\ell)$, where $v_i,v_{i+1}\in e_i$ for each $1\leq i\leq \ell-1$ and $v_1,v_\ell\in e_\ell$. Let $supp($Berge-$C_\ell$)=$\{v_1,v_2,\dots, v_\ell\}$ be the \emph{support  vertices} of Berge-$C_\ell$. 
A \emph{Berge path} Berge-$P_\ell$ of length $\ell$ is an alternating sequence of distinct vertices and hyperedges of the form $(v_1,e_1,v_2,e_2,\dots ,v_\ell,e_\ell,v_{\ell+1})$, where $v_i,v_{i+1}\in e_i$ for each $1\leq i\leq \ell$.
As a generalization of Berge cycles and Berge paths, the notion of Berge hypergraphs is introduced by Gerbner and Palmer \cite{GDPC2017}. Given a graph $F$, a hypergraph $H$ is a \emph{Berge-F} if there exists a bijection $\phi: E(F) \rightarrow E(H)$ such that $e\in \phi(e)$ for all $e\in E(F)$. 

Saturation numbers for (2-uniform) graphs were first studied by Erd\H{o}s, Hajnal and Moon in \cite{EPHAMJ1964}, where they determined $${sat}_2(n, K_{r+1})=\binom{n}{2}-\binom{n-r+1}{2}$$ for any $2\le r\le n-1$. In 1986, K\'{a}szonyi and Tuza \cite{KLTZ1986} showed that $sat_2(\mathcal{F})=O(n)$ for any finite family of graphs. 
Since then, a large quantity of work in this area has been carried out  about the saturation numbers of disjoint union of cliques and generalized friendship graphs \cite{FFGJ2009},
complete multipartite graphs \cite{BFP10, C14}, cycles \cite{C09,FFL97,FK13, GLS06,LSZ15,O72,T89}, and  paths \cite{KLTZ1986}.  Among these results, one can find that the cycles attracted much attention in the saturation problems. 
We now mention some results for cycles with length less than $n$. Note that  $sat_2(n,C_3)=sat_2(n,K_3)$, which  has been determined in \cite{EPHAMJ1964}. 
In 1972, Ollmann \cite{O72} determined that $$sat(n,C_4)=\bigg\lfloor\frac{3n-5}{2}\bigg\rfloor$$ for any $n\ge 5$, a shorter proof was given by Tuza \cite{T89}. In 2009, Chen proved that $$sat(n,C_5)=\bigg\lceil\frac{10n-10}{7}\bigg\rceil$$ for any $n\ge 21$. Other than these results, no result of short cycles on the exact value is known.


The saturation problems in hypergraphs are also studied extensively. 
Confirming a conjecture of Tuza \cite{T86,T88}, Pikhurko \cite{P99a} proved that $sat_k(n,\mathcal{F})=O(n^{k-1})$ for any finite family $\mathcal{F}$ of $k$-uniform hypergraphs. For any given graph $F$,
it was conjectured   that $sat_k(n,$Berge-$F)=O(n)$ in \cite{ESGPGN2019}, which was confirmed for $3\le k\le 5$ in \cite{ESGPGN2019a}, but is still open in general.  We are also interested in the research of cycles in   hypergraphs. The cycles in hypergraphs are the Begre-cycles.
In 2019, the authors of \cite{ESGPGN2019} determined $sat_k(n,$Berge-$K_3)=\big\lceil\frac{n-1}{k-1}\big\rceil$ for any $n\ge k+1$ and provided some upper bounds for general Berge-cycles. However, there is no any nontrivial lower bound for any Berge-cycles with length more than 3.

In addition to the saturation number of hypergraphs, the linear extremal problems have attracted significant attention. Here we concentrate on the linear extremal number of Berge-$C_3$ and  Berge-$C_4$. 
It is wildly considered that 
determining  the linear extremal number of a Berge-$C_3$ in 3-uniform hypergraphs is equivalent to the famous (6,3)-problem, one of the problems of  Brown, Erd\H{o}s, and S\'os in \cite{BWGEPSV1973}. 
In 1976, Ruzsa and Szemer\'{e}di  \cite{RIZSE1978} proved that 
$$n^{2-\frac{c}{\sqrt{log~ n}}}<ex_3^{lin}(n,\text{Berge-}C_3)=o(n^2)$$
for some constant $c>0$. For Berge-$C_4$, the authors of \cite{LFVJ2003} and \cite{EBGEMA2019} given $\frac{n^{3/2}}{6}+O(n)$ as the lower bound and upper bound for $ex_3^{lin}(n,$Berge-$C_4$), respectively. That is $$ ex_3^{lin}(n,\text{Berge-}C_4)=\frac{n^{3/2}}{6}+O(n).$$
Analogous to the linear extremal problem of Berge-cycles, the following question is a natural dual  of this problem in saturation problems. 

\begin{question}\label{linsat}(Linear saturation problem)
Let $t\ge 3$ and $k\ge 3$. 
What is the minimum number of hyperedges in linear Berge-$C_t$-saturated $k$-uniform hypergraphs? 
\end{question}

In this paper, we investigate the above question and answer 
Question \ref{linsat} for $t=3$ and $k\ge 3$; for $t=4$,  we provide an upper bound and when the linear  Berge-$C_4$-saturated $3$-hypergraph is disconnected, we give a lower bound. Our results are  as follows. 

\begin{theorem}\label{ctlower}
Let $t\ge 3$, $k\ge 3$, and $n\ge k$. Then
$${sat}^{lin}_k(n,\text{Berge-}C_t)\ge  \bigg\lfloor\frac{n-1}{k-1}\bigg\rfloor.$$
\end{theorem}

\begin{theorem}\label{c3} 
Let $n\geq 6$ and $k\ge 3$. Then
$${sat}^{lin}_k(n,\text{Berge-}C_3)= \bigg\lfloor\frac{n-1}{k-1}\bigg\rfloor.$$
\end{theorem}

For a Berge-$C_4$, we have the following results in $3$-uniform hypergraphs. 
\begin{theorem}\label{upc4} 
Let $n\geq 1$.
Then
$$\operatorname{sat}^{lin}_3(n, \text{Berge-}C_4)\leq \left\{
\begin{array}{lcl}
	\lfloor\frac{5(n-1)}{6}\rfloor, n\equiv1, 15, 16    &&{\pmod{18}};\\
	\lfloor\frac{5(n-2)}{6}\rfloor, n\equiv0, 2, 3, 4, 5, 6, 7, 17 &&{\pmod{18}};\\
	\lfloor\frac{5(n-3)}{6}\rfloor, n\equiv8, 9, 10, 11, 12, 13    &&{\pmod{18}};\\
	\lfloor\frac{5(n-4)}{6}\rfloor, n\equiv14    &&{\pmod{18}}.\\
\end{array}\right.
$$
\end{theorem}


%


\begin{theorem}\label{lwc4} Let 
$H$ be a  Berge-$C_4$-saturated linear $3$-uniform hypergraph with $n\ge 6$ vertices  and be disconnected. 
For any component $H'$ of $H$ with $n'$ vertices, 
then $$e(H')\geq \begin{cases}
	0, &\text{ when $\delta(H')= 0$;}\\
	\frac{2n'}{3} -2,  &\text{ when $\delta(H')= 1$;}\\
	\frac{13n'-29}{18}, &\text{ when $\delta(H') = 2$;}\\
	n', &\text{ when $\delta(H')\ge 3$.}
\end{cases} $$
Furthermore, we have $ e(H)\ge \frac{2n }{3}-4.$	
\end{theorem}

The rest of the paper is outlined as follows. 
In Section \ref{s2}, we give some basic notations and preliminaries. We devote Section \ref{s3} to prove Theorems \ref{ctlower} and \ref{c3}. 
An upper bound of the linear saturation number of Berge-$C_4$ is given by constructing a Berge-$C_4$-saturated linear 3-uniform hypergraph in Section \ref{s4}. Finally we obtain a lower bound of the number of hyperedges for  any  disconnected Berge-$C_4$-saturated linear $3$-uniform hypergraph in Section \ref{s5}.   

\section{Preliminaries}\label{s2}
Let us show some definitions and symbols which will be used in the proof. For a hypergraph $H$, let $\delta_H$ denote the minimum vertex degree of $H$. Denote by $[n]$ the  integer set $\{1,2,\dots,n\}$. For $X\subset V(H)$, denote by $H[X]$  the induced hypergraph of $H$, whose vertex set is $X$ and whose hyperedge set is the set of those hyperedges of $H$ that have all vertices in $X$. For a graph $G$ with  $X,Y\subset V(G)$, let $G[X,Y]$ be  a bipartite graph,  whose vertex set is $X\cup Y$ and whose  edge set is the set of those  edges of $G$ that have one vertex in $X$ and the other one  in $Y$.
\begin{definition}
\rm 	A linear $k$-uniform hypergraph is a star if   all hyperedges of it  intersect at exactly one vertex.
\end{definition}
\begin{definition}\label{P3connect}
\rm 	A linear $k$-uniform hypergraph is Berge-$P_\ell$-connected if there is a Berge-$P_\ell$ between any two nonadjacent  vertices.
\end{definition}
By Definition \ref{P3connect}, if a connected hypergraph consists of   an  isolated vertex or   one hyperedge, then it is   Berge-$P_\ell$-connected.

\begin{definition}\label{identify}
\rm	Given two hypergraphs $G$ and $H$ (it is possible that $G=H$), denote by $G\cdot H(v)$ the hypergraph obtained from the disjoint union $G\cup H$ by identifying the two vertices $v$ in $G$ and $H$ as one. For $H=G$, let $G\cdot G(v)=2G(v)$. 
\end{definition}
\rm By Definition \ref{identify}, we have  $V(G\cdot H(v))=(V(G)\setminus\{v\})\cup V(H)$, $|V(G\cdot H(v))|=|V(G)|+|V(H)|-1$ and $E(G\cdot H(v))=E(G)\cup E(H)$. 

\begin{definition}\label{identifyk}
\rm	For any integer $k\ge 0$, let $kG\cdot H(v)$ be the hypergraph obtained from $(\cup_{0\le i\le k}G)\cup H$ by identifying the $k+1$  vertices $v$ in $k$ copies of $G$ and $H$ as one. For $k=0$, let $kG\cdot H(v)=H$. 
\end{definition}
We have  $V(kG\cdot H(v))=\big[\cup_{0\le i\le k}(V(G)\setminus\{v\})\big]\cup V(H)$, $|V(kG\cdot H(v))|=k|V(G)|+|V(H)|-k$ and $E(kG\cdot H(v))=(\cup_{0\le i\le k}E(G))\cup E(H)$.

\section{The proofs of Theorems \ref{ctlower} and \ref{c3}}\label{s3}
\noindent{\bf Proof of Theorem \ref{ctlower}.}
Let $H$ be a Berge-$C_t$-saturated linear $k$-uniform hypergraph with $\ell$  components, denoted by $H_i$, for $i\in [\ell]$. We have  $\ell \leq k-1$, otherwise if there are $k$ components, we can choose a vertex from each component to form a hyperedge, the addition of which will not create a copy of Berge-$C_t$, a contradiction. 
To obtain a lower bound of $\operatorname{sat}^{lin}_k$($n$, Berge-$C_t$), we consider a component of $H$. 

\begin{claim}\label{component}
Each component of $H$ has at least $\frac{n-1}{k-1}$ hyperedges.
\end{claim}  
\noindent{\textit{Proof.}} Let us choose a component $H'$ with $n'$ vertices and $m'$ hyperedges. If $m'=0$, then $H'$ is an isolated vertex, Claim \ref{component} holds.
If $m'\geq 1$, we construct a bipartite graph $G'=(A,B)$ with $|A|=m$ and $|B|=n$, where $A=\{e: e\in E(H')\}$, $B=\{v: v\in V(H')\}$ and $E(G')=\{ev: e\in A$ and $v\in B$ with $v\in e$\}. 
For any two vertices of $B$, there is  a path between them  in $G'$, by the connectivity of $H'$.  For any $e\in A$, there exists some vertex $v\in B$ such that $e$ and $v$ are adjacent in $G'$. Since any two vertices in $B$ has a path between them in $G'$, there is a path between $e$ and $u$ for any $u\in B$. It follows that there is a path between $e$ and $e'$ for any $e,e'\in A$. Thus $G'$ is connected, and so  $G'$ has at least $|V(G')|-1=n'+m'-1$ edges. By the definition of $G'$, the number of edges of $G'$ is $km'$.
Hence, $km'\geq n'+m'-1$ and   $m'\geq \frac{n'-1}{k-1}.$ 
$\hfill\qedsymbol$	

For  each component $H_i$ with $n_i$ vertices and $m_i$ hyperedges, we have $m_i\geq \frac{n_i-1}{k-1}$ by Claim \ref{component}. The number of hyperedges of $H$ is $m=m_1+m_2+\dots +m_\ell$. We can see $$m_1+m_2+\dots +m_\ell\geq\frac{n_1-1}{k-1}+\frac{n_2-1}{k-1}+\dots +\frac{n_\ell-1}{k-1}=\frac{n-1-(\ell-1)}{k-1}\ge \frac{n-1-(k-2)}{k-1}.$$
Let  $n-1\equiv y {\pmod {(k-1)}}$. Then $\frac{n-1-(k-2)}{k-1}=\lfloor\frac{n-1}{k-1}\rfloor+\frac{y-(k-2)}{k-1}.$ Since $0\le y\le k-2$, $\lceil\frac{y-(k-2)}{k-1}\rceil=0.$ Note that $m\geq \lfloor\frac{n-1}{k-1}\rfloor+\frac{y-(k-2)}{k-1}$ and $m$ is an integer, then $m\geq \lfloor\frac{n-1}{k-1}\rfloor+\lceil\frac{y-(k-2)}{k-1}\rceil=\lfloor\frac{n-1}{k-1}\rfloor.$  Hence,  $\operatorname{sat}^{lin}_k$($n$, Berge-$C_t) \ge \lfloor\frac{n-1}{k-1}\rfloor.$ 
$\hfill\qedsymbol$

\medskip

We have completed the proof of Theorem \ref{ctlower}. Next we prove Theorem \ref{c3}. 

\medskip

\noindent{\bf Proof of Theorem \ref{c3}.}
By Theorem \ref{ctlower}, we have $\operatorname{sat}^{lin}_k$($n$, Berge-$C_3) \ge \lfloor\frac{n-1}{k-1}\rfloor.$ It remains  to prove $\operatorname{sat}^{lin}_k$($n$, Berge-$C_3) \le \lfloor\frac{n-1}{k-1}\rfloor.$
For the upper  bound,  we construct a hypergraph $G$ consisted of a star with $\lfloor\frac{n-1}{k-1}\rfloor$ hyperedges and $t$ isolated vertices, where  $n-1\equiv t\pmod {k-1}$ and $t\le k-2$. 
Note that $G$ contains no Berge-$C_3$ and if we add a new hyperedge, then it contains at least two vertices from the star, which will create a Berge-$C_3$.
Thus $G$ is Berge-$C_3$-saturated with $\lfloor\frac{n-1}{k-1}\rfloor$ hyperedges. Then $\operatorname{sat}^{lin}_k$($n$, Berge-$C_3$)$\le e(G)= \lfloor\frac{n-1}{k-1}\rfloor.$ Finally, we have	$${sat}^{lin}_k(n,\text{Berge-}C_3)= \bigg\lfloor\frac{n-1}{k-1}\bigg\rfloor.$$
$\hfill\qedsymbol$
\newpage
\section{The proof of Theorem \ref{upc4}}\label{s4}
To prove the theorem, we need to construct a  Berge-$C_4$-saturated linear 3-uniform hypergraph with $n$ vertices. 
If a hypergraph $H$ is Berge-$P_3$-connected and contains no Berge-$C_4$, then $H$ is Berge-$C_4$-saturated by the definition of Berge-$C_4$-saturated hypergraphs. 
We construct a hypergraph $T^*$ with 15 hyperedges and $19$ vertices in Figure \ref{T*}. 
We first show $T^*$ contains no Berge-$C_4$. 
By contradiction, suppose  
there is a Berge-$C_4$ in $T^*$, denoted by $\mathcal{C}_4$. 
Let $U=\{v_2,v_4,v_6,v_8,v_{10},v_{12},v_{14},v_{16},v_{18}\}$ and $W=\{v_1,v_3,v_5,v_7,v_{9},v_{11},v_{13},v_{15},v_{17}\}$.

\begin{figure}[H]
\centering{
	\begin{minipage}[t]{0.48\textwidth}
		\centering
		\includegraphics[scale=0.60]{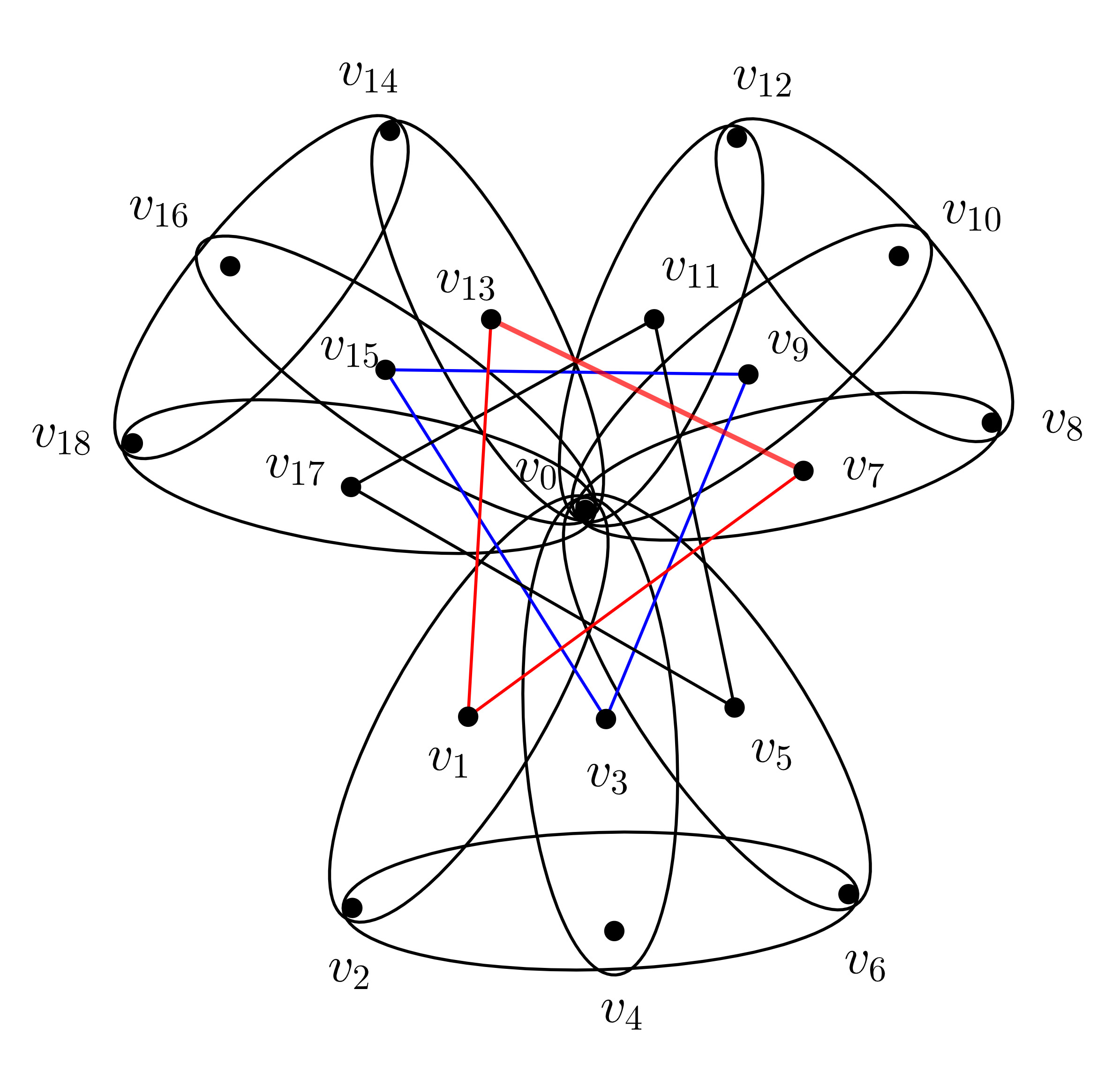}
		\caption*{$T^*$} 
	\end{minipage}
	\begin{minipage}[t]{0.38\textwidth}
		\centering
		\includegraphics[scale=1.0]{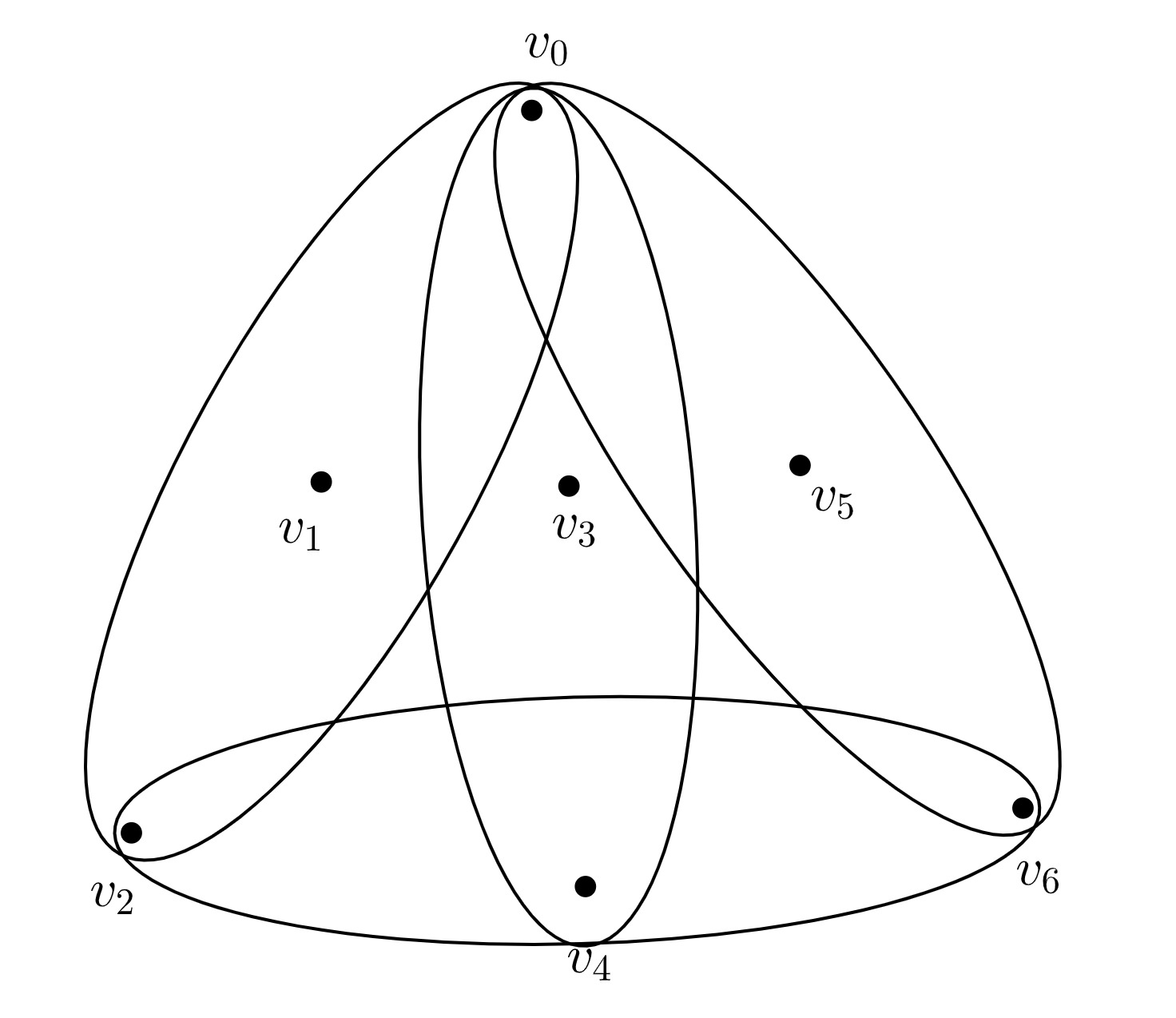}
		\caption*{$T'$} 
	\end{minipage}  
}
\caption{The hypergraphs $T^*$ and $T'$ }\label{T*}
\end{figure}

First, we  consider the case   $v_2\in supp( \mathcal{C}_4)$ as the first support vertex. Since $d(v_2)=2$, the two hyperedges in the $\mathcal{C}_4$ containing $v_2$ are 
$\{v_2,v_4,v_6\}$ and $\{v_2,v_1,v_0\}$. The second support vertex is $v_4$ or $v_6$ and the fourth support vertex is $v_1$ or $v_0$. Without loss of generality, we choose $v_4$ as the second support vertex.   Then the second hyperedge is $\{v_4,v_3,v_0\}$ by $d(v_4)=2$. The third support  vertex is $v_3$ or $v_0$. But there is no new hyperedge containing the third and fourth support vertices as the hyperedges in the sequence of Berge-$C_4$ are distinct. Hence, the vertex $v_2\notin supp(\mathcal{C}_4)$. 
By the symmetry of $T^*$, we have $U\cap supp(\mathcal{C}_4)=\emptyset$ 
and so
the hyperedges likes $\{v_2,v_4,v_6\},\{v_8,v_{10},v_{12}\},$ and $\{v_{14},v_{16},v_{18}\}$ are not in   $\mathcal{C}_4$.
Next we assume $v_1\in supp(\mathcal{C}_4)$ as the first support vertex. By $d(v_1)=2$, 
the two  hyperedges containing it are $\{v_0,v_1,v_2\}$ and $\{v_1,v_7,v_{13}\}$. The second support vertex is $v_0$ by $v_2\notin supp(\mathcal{C}_4)$. By symmetry, we may assume that 
the fourth support  vertex is $v_7$. It follows that 
the third hyperedge is $\{v_0,v_7,v_8\}$ from $d(v_7)=2$, and the third support vertex is exactly $v_8$, which contradicts $v_8\notin supp(\mathcal{C}_4)$.  Then $v_1\notin supp(\mathcal{C}_4)$ and $W\cap supp(\mathcal{C}_4)=\emptyset$ by symmetry. 
We can see   $T^*$ contains no copy of Berge-$C_4$. 

We construct a hypergraph $T'$ in Figure \ref{T*}. Since $T'$ is a subgraph of $T_1$, $T'$ contains no Berge-$C_4$. Observe that $v_1$ has a Berge-$P_3$ to $v_3$, that is, $(v_1,\{v_0,v_1,v_2\}, v_2,\{v_2,v_4,v_6\},v_4,\\\{v_4,v_3,v_0\},v_0)$ and $v_1$ has a Berge-$P_3$ to $v_4$, that is,  $(v_1,\{v_0,v_1,v_2\}, v_0,\{v_0,v_5,v_6\},v_6,\\\{v_6,v_4,v_2\},v_4)$. 
By symmetry, $T'$ is Berge-$P_3$-connected. Hence, $T'$ is Berge-$C_4$-saturated.
Let us turn to $T^*$. We can see that $v_0$ is adjacent to all other vertices in $T^*$ and has a Berge-$P_2$ to them. 
Sicne $T'$ is Berge-$P_3$-connected, $T^*$ is Berge-$P_3$-connected and so   $T^*$ is Berge-$C_4$-saturated.

Let $T_1=\big\lfloor\frac{n}{18}\big\rfloor T^*(v_0)$. 
Since each vertex from $T^*$ has a Berge-$P_2$ to $v_0$ and is adjacent to $v_0$, $T_1$ is Berge-$P_3$-connected by $T^*$ is Berge-$P_3$-connected. For each copy of $T^*$, it contains no Berge-$C_4$, and $v_0$ is a cut vertex for $T_1$. Thus $T_1$ contains no Berge-$C_4$ and  $T_1$ is a Berge-$C_4$-saturated hypergraph.

\begin{claim}\label{T1T'}
The hypergraph $T_1\cdot T'(v_0)$
is also Berge-$P_3$-connected and contains no   Berge-$C_4$.
\end{claim}

\begin{proof}
Since	both $T_1$ and $T'$ contain no Berge-$C_4$, the identifying of $T_1$ and  $T'$   also  contains no Berge-$C_4$ by $v_0$ is a cut vertex. 
It remains to verify that any two nonadjacent vertices have a Berge-$P_3$ between them. 
Note that  $T_1$ and $T'$ are Berge-$P_3$-connected.  We only need to consider the pairs of vertices, which one is from $T_1\setminus\{v_0\}$ and the other is from $T'\setminus\{v_0\}$. 
Recall each vertex in $T_1$ has a Berge-$P_2$ to $v_0$ and each vertex in $T'$ is adjacent to $v_0$. Hence,  any pair of vertices, which one is from  $T_1\setminus\{v_0\}$ and the other  is from $T'\setminus\{v_0\}$, has a Berge-$P_3$ connecting them. 
\end{proof}


\begin{claim}\label{twocompo}
The disjoint union of two Berge-$P_3$-connected hypergraphs containing no Berge-$C_4$ is Berge-$C_4$-saturated.  
\end{claim}

\noindent{\textit{Proof.}} We can see  the disjoint union contains no Berge-$C_4$. Adding any new hyperedge, it must contain at least two vertices, saying $v_1, v_2$, from the same Berge-$P_3$-connected hypergraph. By the linearity of the hypergraph, $v_1$ and $v_2$ are not adjacent. Since the hypergraph is Berge-$P_3$-connected, it will create a Berge-$C_4$. Thus Claim \ref{twocompo} holds.  
$\hfill\qedsymbol$

\medskip

For $i\in \{2,3,4,\dots,18\}$, 	we construct a hypergraph $T'_i$ consists of $v_0$ and $i-1$ new vertices such that there is at most one hyperedge in $T'_i$ containing vertices which are not adjacent to $v_0$. Let    $T_i=\big\lfloor\frac{n}{18}\big\rfloor T^*\cdot T_i'(v_0)=T_1\cdot T'_i(v_0)$ for $i\in \{2,3,4,\dots,18\}$. The specific constructions of $T_i$ are as follows in Figures \ref{T2481014},  \ref{T35}, \ref{T67911}, \ref{T12356}, and \ref{T178}.
When $n\equiv 2,4,8,10,14\pmod{18}$, we have $T_i$ is Berge-$C_4$-saturated for $i\in\{2,4,8,10,14\}$ by Claims \ref{T1T'}, \ref{twocompo}.

\begin{claim}\label{inducedgraph}
If $T'_i$ is Berge-$C_4$-saturated for $i\in \{3,5,6,7,9,11,12,13,15,16,17,18\}$, then $T_i$ is Berge-$C_4$-saturated.
\end{claim}	 

\begin{proof}
Since both $T_1$ and $T'_i$ contain no Berge-$C_4$, the identifying of $T_1$ and  $T'_i$  also  contains no Berge-$C_4$ by $v_0$ is a cut vertex. 
Recall each vertex in $T_1$ is adjacent to $v_0$ and each vertex in $T_1$ has a Berge-$P_2$ to $v_0$.  We can see that each vertex in $T'_i$ is adjacent to $v_0$ or has a Berge-$P_2$ to $v_0$. Hence, any pair of vertices, which one is from  $T_1\setminus\{v_0\}$ and the other is from $T'_i\setminus\{v_0\}$, has a Berge-$P_3$.  Additionally, $T_1$ and $T'_i$ are Berge-$C_4$-saturated, so any hyperegde $e\in
\binom{V(T_i)}{3}\setminus E(T_i)$ such that $E(T_i)\cup \{e\}$ is linear will create a Berge-$C_4$. By the definition,  $T_i$ is Berge-$C_4$-saturated.
\end{proof}

\begin{figure}[H]
\centering{
	\begin{minipage}[t]{0.48\textwidth}
		\centering
		\includegraphics[scale=0.70]{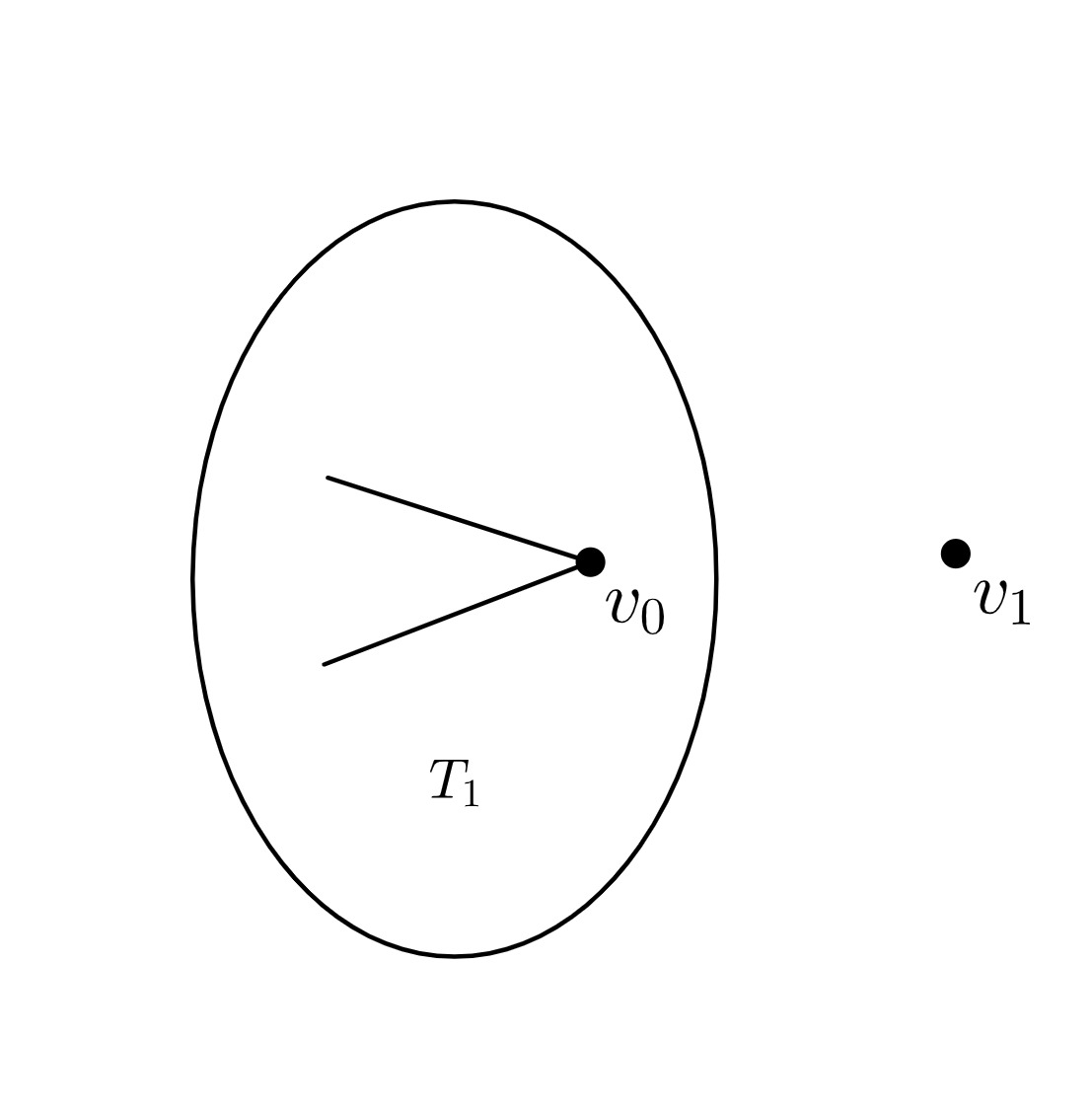}
		\caption*{$T_{2}$}
	\end{minipage}
	\begin{minipage}[t]{0.48\textwidth}
		\centering
		\includegraphics[scale=0.70]{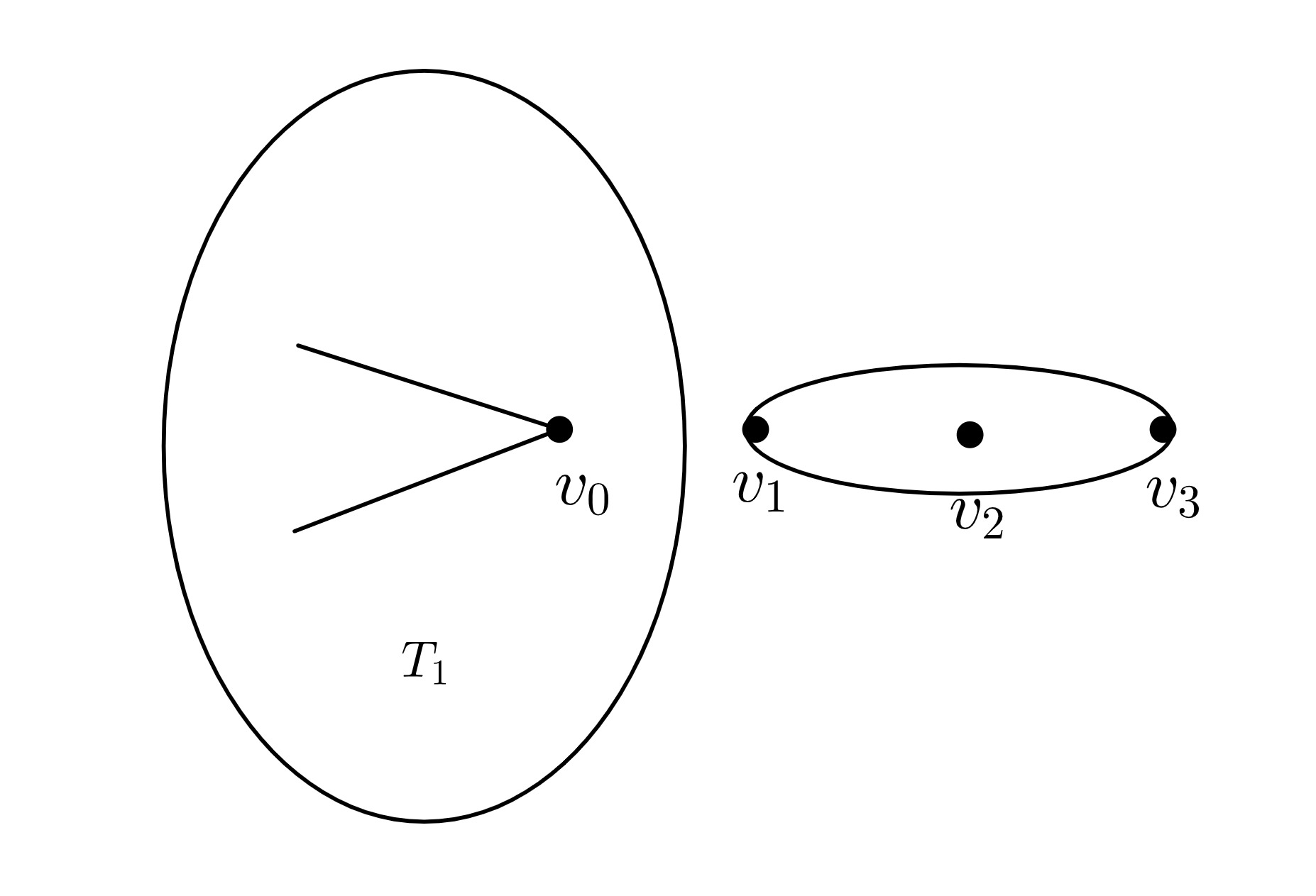}
		\caption*{$T_{4}$}
	\end{minipage}
	\begin{minipage}[t]{0.48\textwidth}
		\centering
		\includegraphics[scale=0.70]{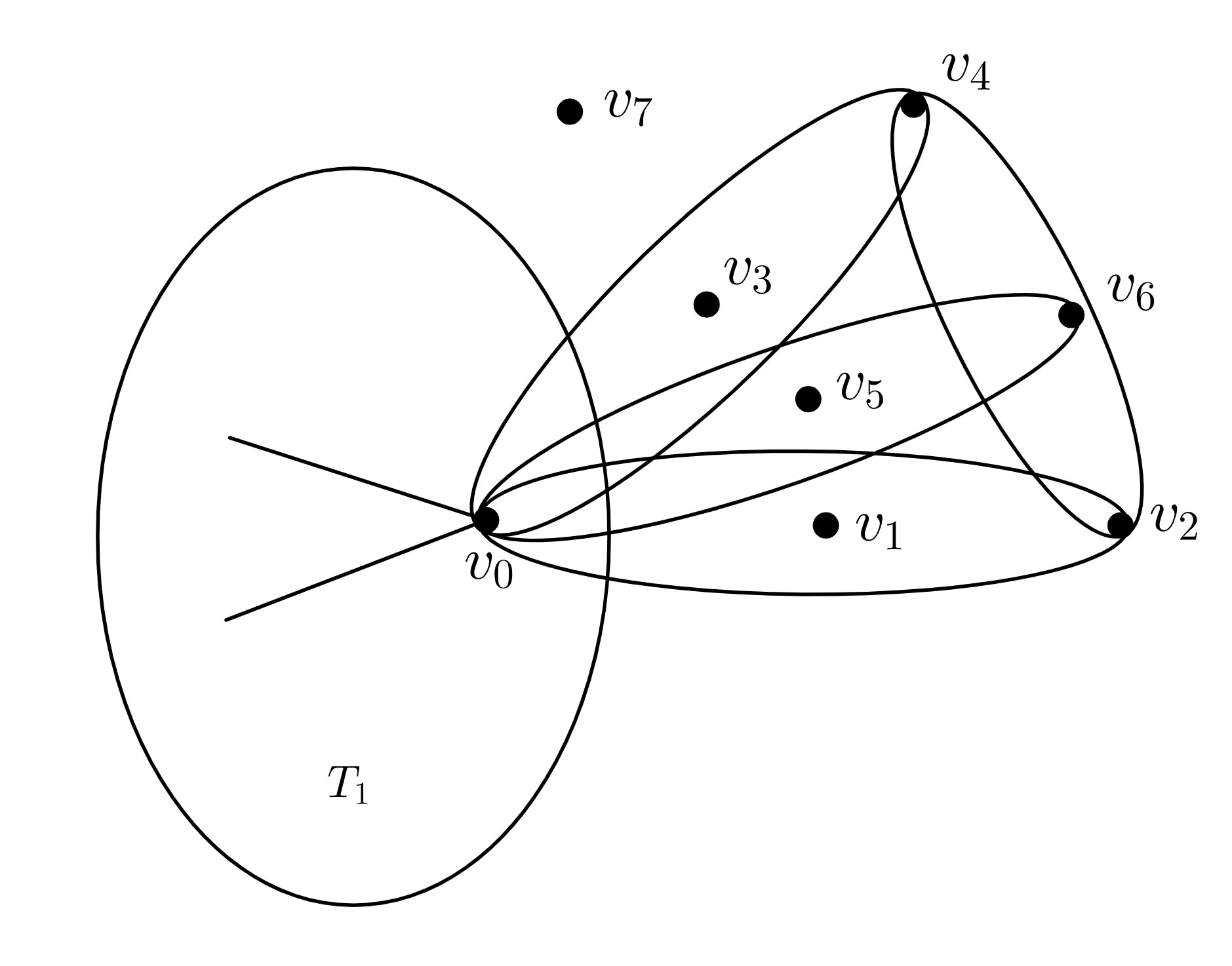}
		\caption*{$T_{8}$}
	\end{minipage}
	\begin{minipage}[t]{0.48\textwidth}
		\centering
		\includegraphics[scale=0.70]{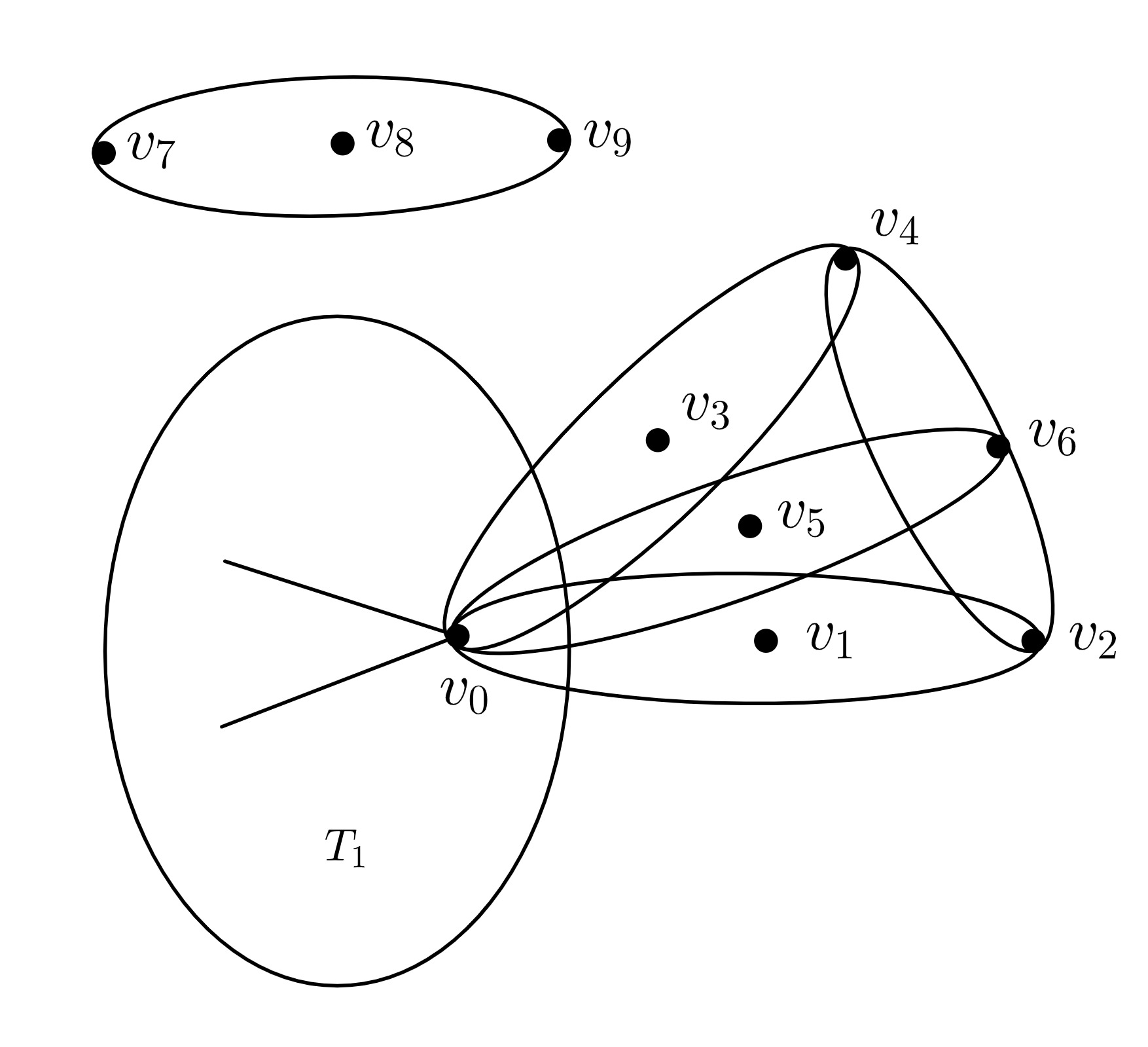}
		\caption*{$T_{10}$}
	\end{minipage}
}
\end{figure}
\begin{figure}[H]
\centering{
	\begin{minipage}[t]{0.48\textwidth}
		\centering
		\includegraphics[scale=0.70]{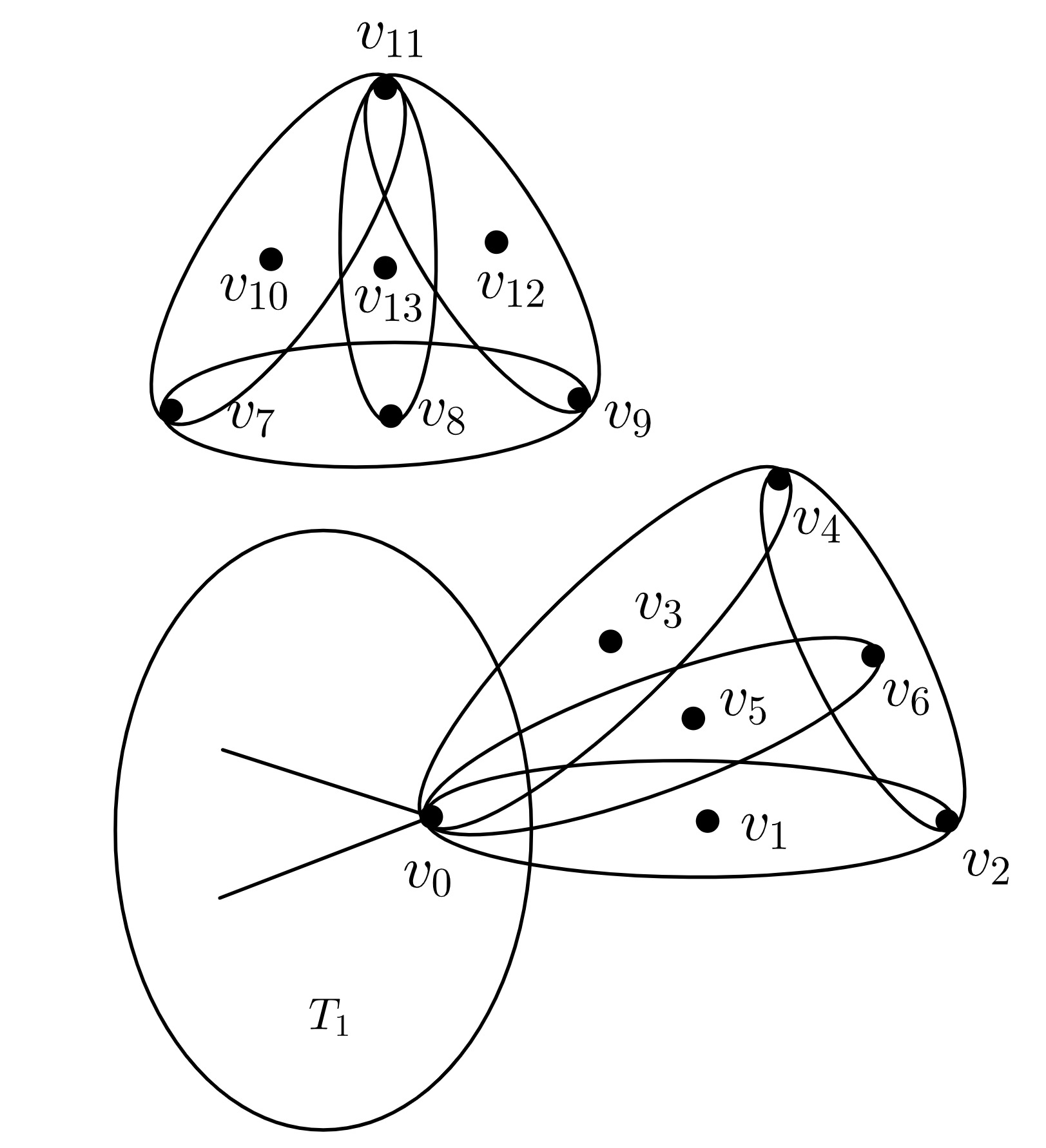}
		\caption*{$T_{14}$}
\end{minipage}	}	
\caption{The hypergraphs $T_2$, $T_4$, $T_8$, $T_{10}$, and $T_{14}$}\label{T2481014}
\end{figure}

\begin{figure}[H]
\centering{
	\begin{minipage}[t]{0.48\textwidth}
		\centering
		\includegraphics[scale=0.70]{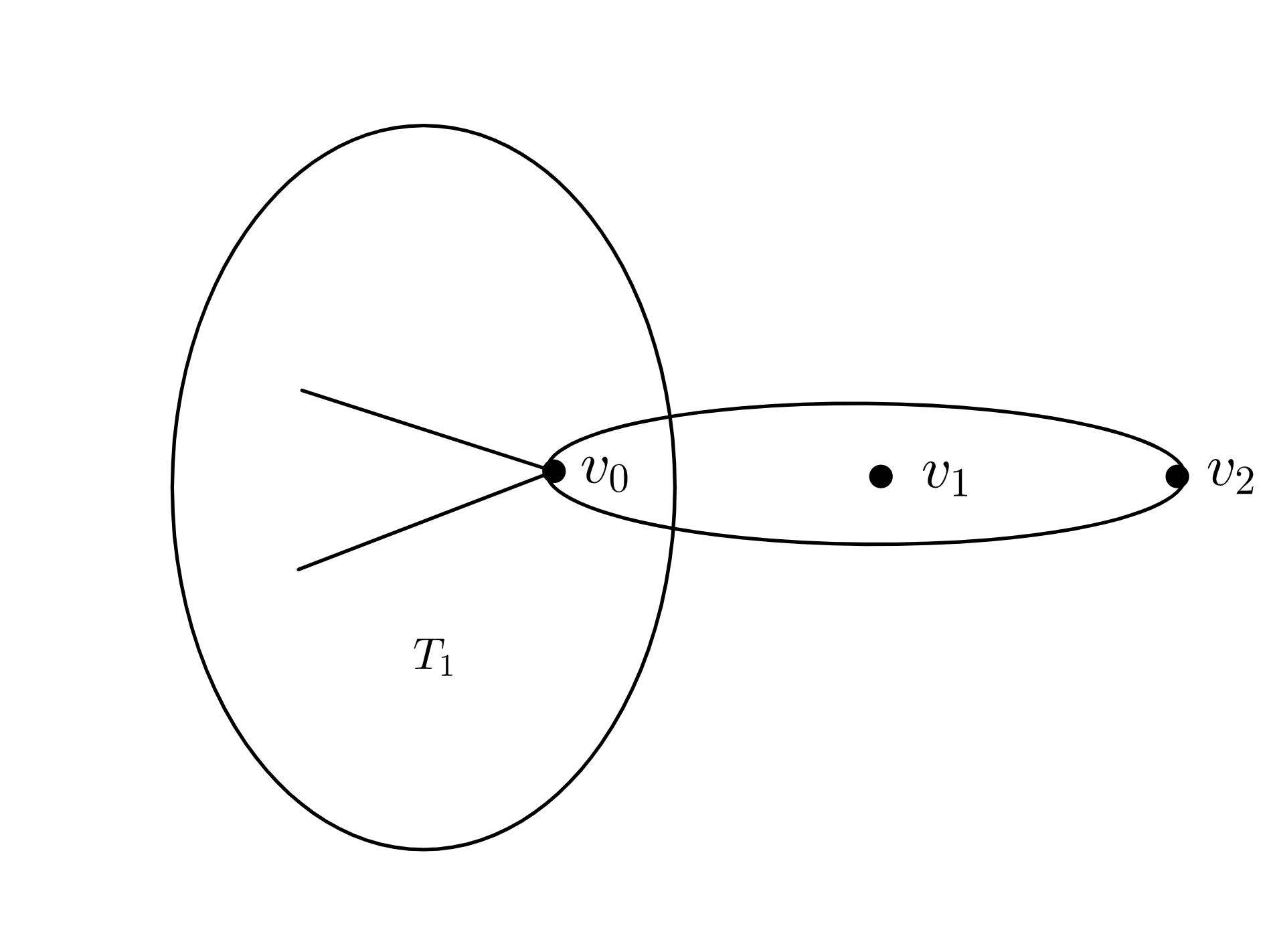}
		\caption*{$T_{3}$}
	\end{minipage}
	\begin{minipage}[t]{0.48\textwidth}
		\centering
		\includegraphics[scale=0.70]{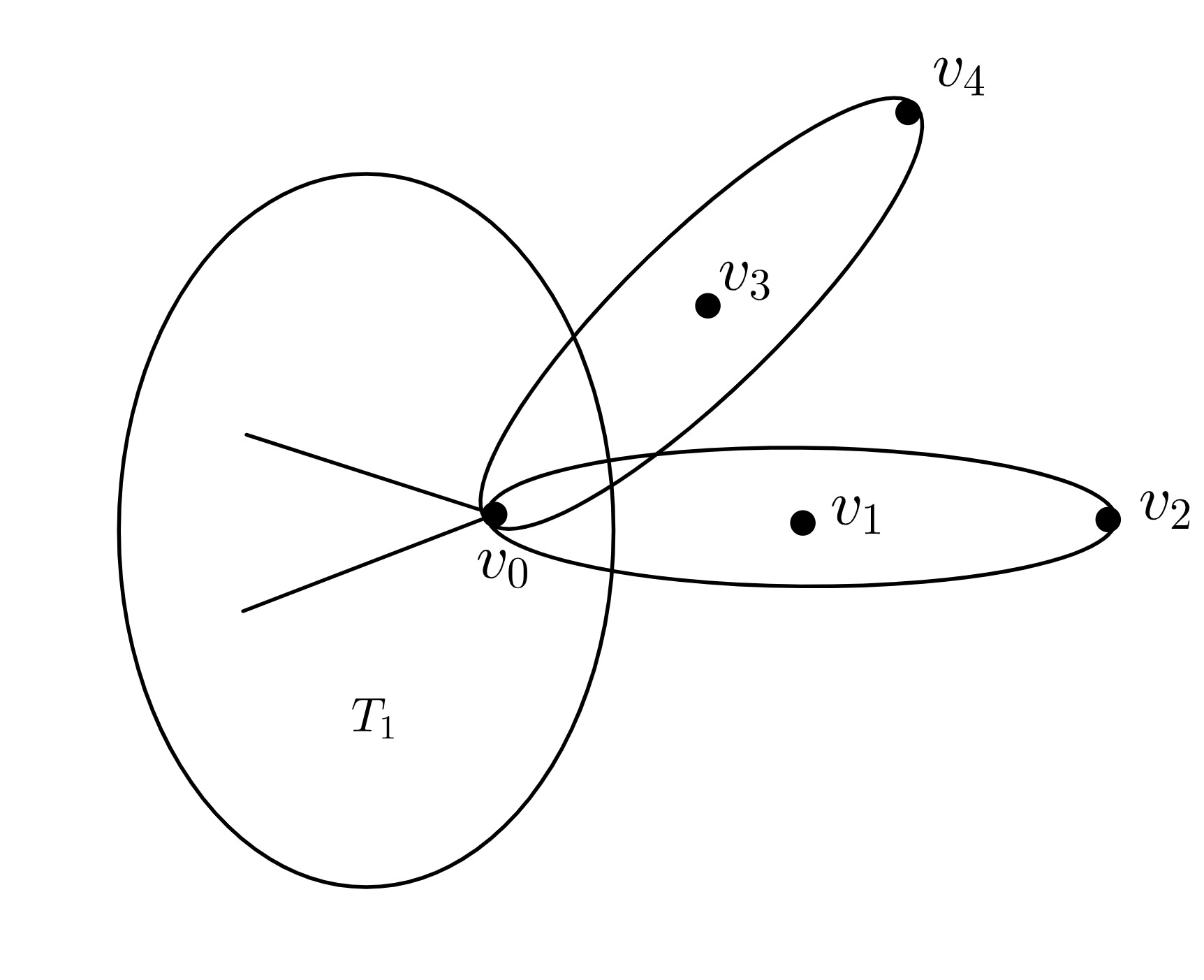}
		\caption*{$T_{5}$}
	\end{minipage}		
}
\caption{The hypergraphs $T_3$ and $T_5$}\label{T35}
\end{figure}

When $n\equiv 3,5\pmod{18}$, 
for $i\in\{3,5\}$, $T'_i$  has at most $5$ vertices and the number of hyperedges is maximal subject to the linearity of $T'_i$. 
Since a Berge-$C_4$ contains at least $6$ vertices, $T_i'$ contains no Berge-$C_4$. Hence, $T'_i$ is Berge-$C_4$-saturated, and so $T_i$ is Berge-$C_4$-saturated for $i\in\{3,5\}$ by Claim \ref{inducedgraph}.  

\begin{figure}[H]
\centering{
	\begin{minipage}[t]{0.48\textwidth}
		\centering
		\includegraphics[scale=0.70]{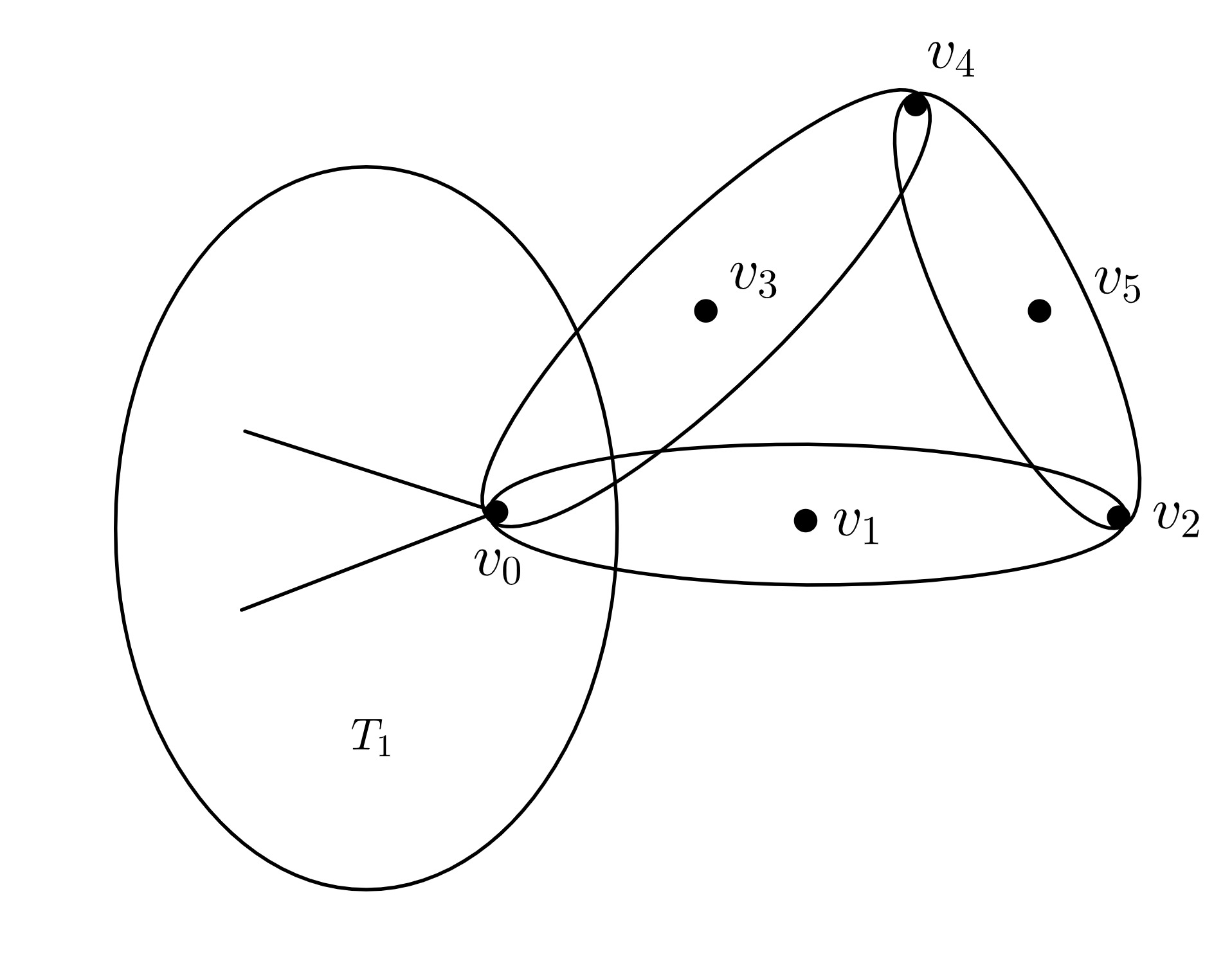}
		\caption*{$T_{6}$}
	\end{minipage}
	\begin{minipage}[t]{0.48\textwidth}
		\centering
		\includegraphics[scale=0.70]{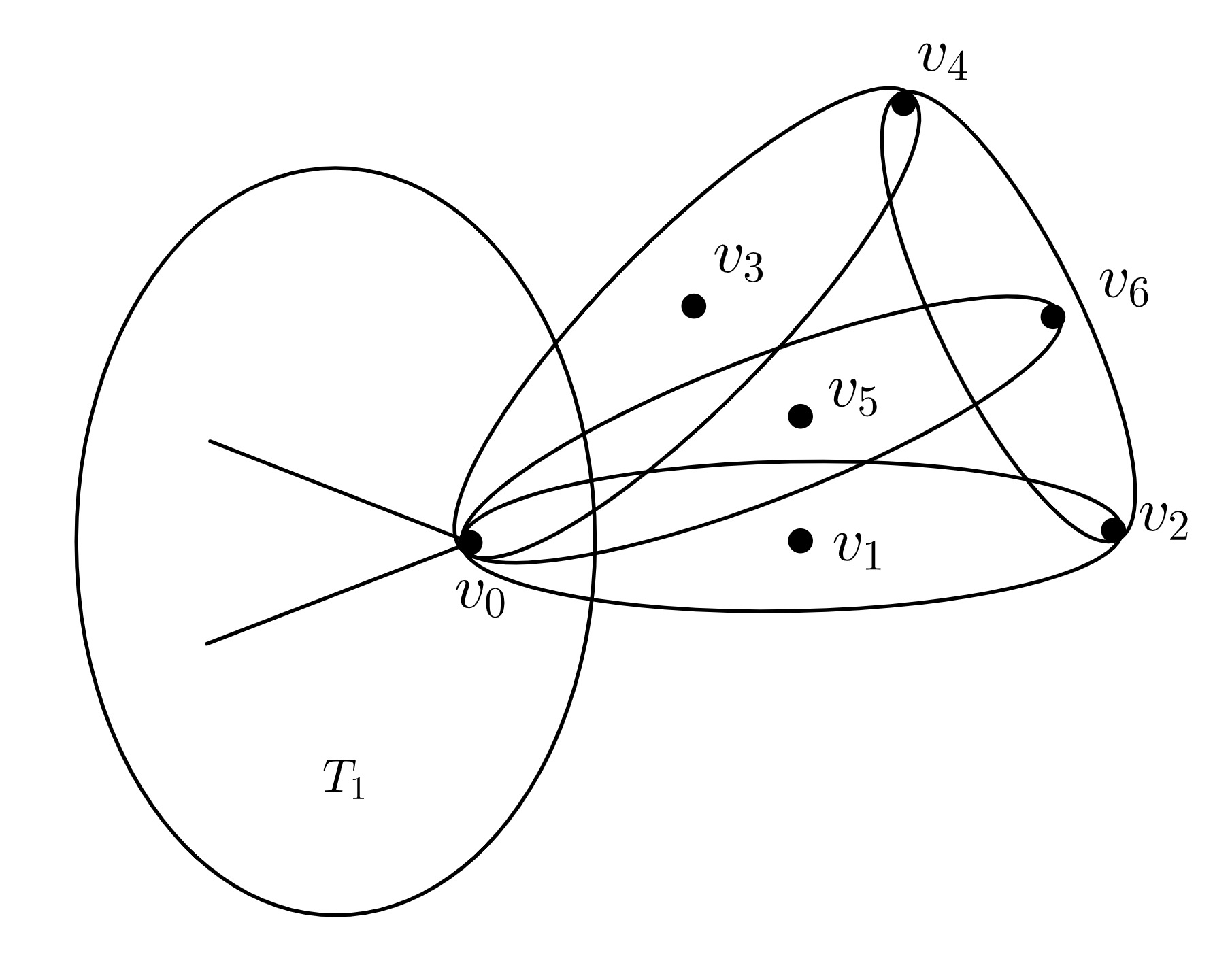}
		\caption*{$T_{7}$}
	\end{minipage}
	\begin{minipage}[t]{0.48\textwidth}
		\centering
		\includegraphics[scale=0.70]{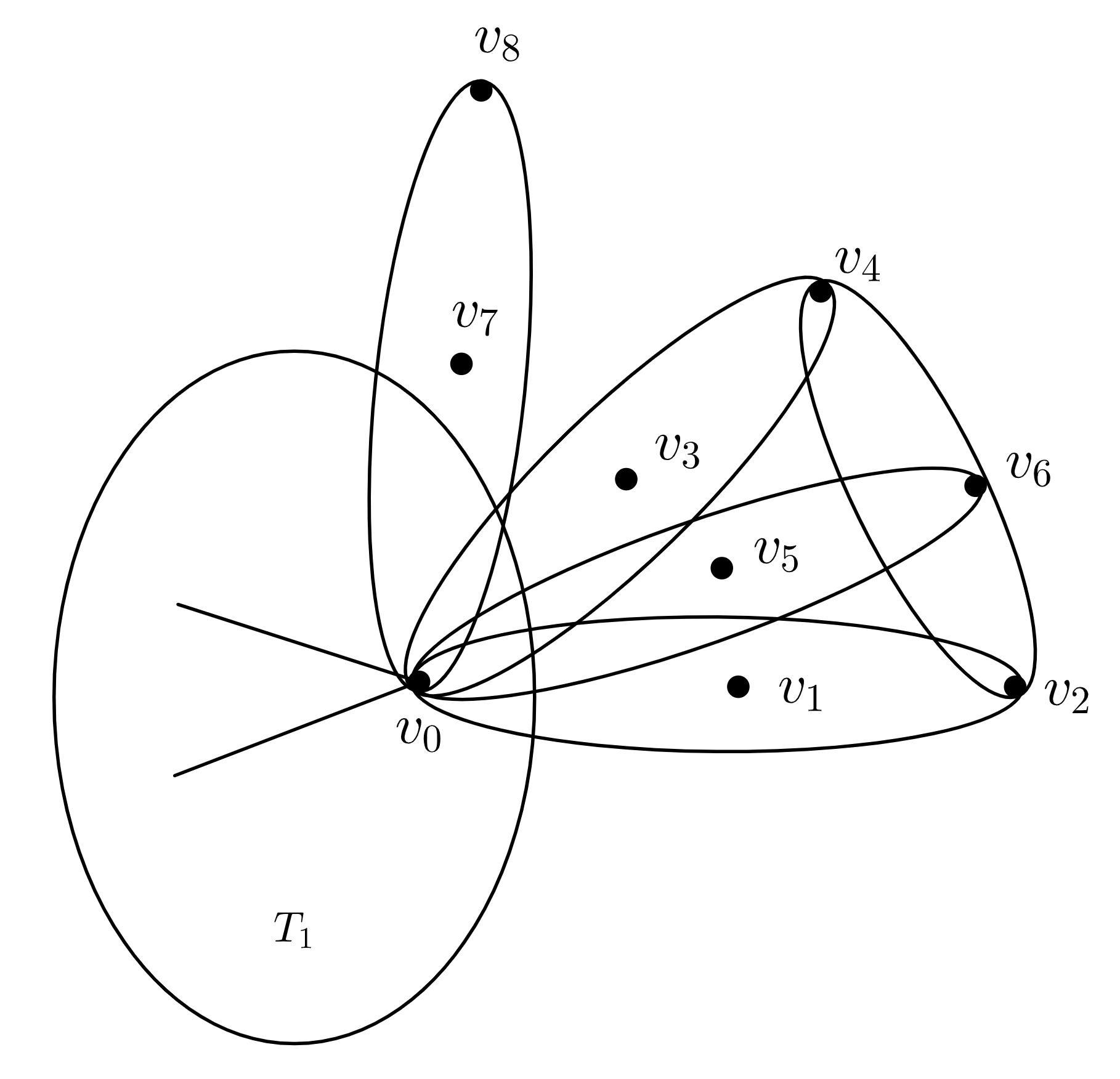}
		\caption*{$T_{9}$}
	\end{minipage}
	\begin{minipage}[t]{0.48\textwidth}
		\centering
		\includegraphics[scale=0.70]{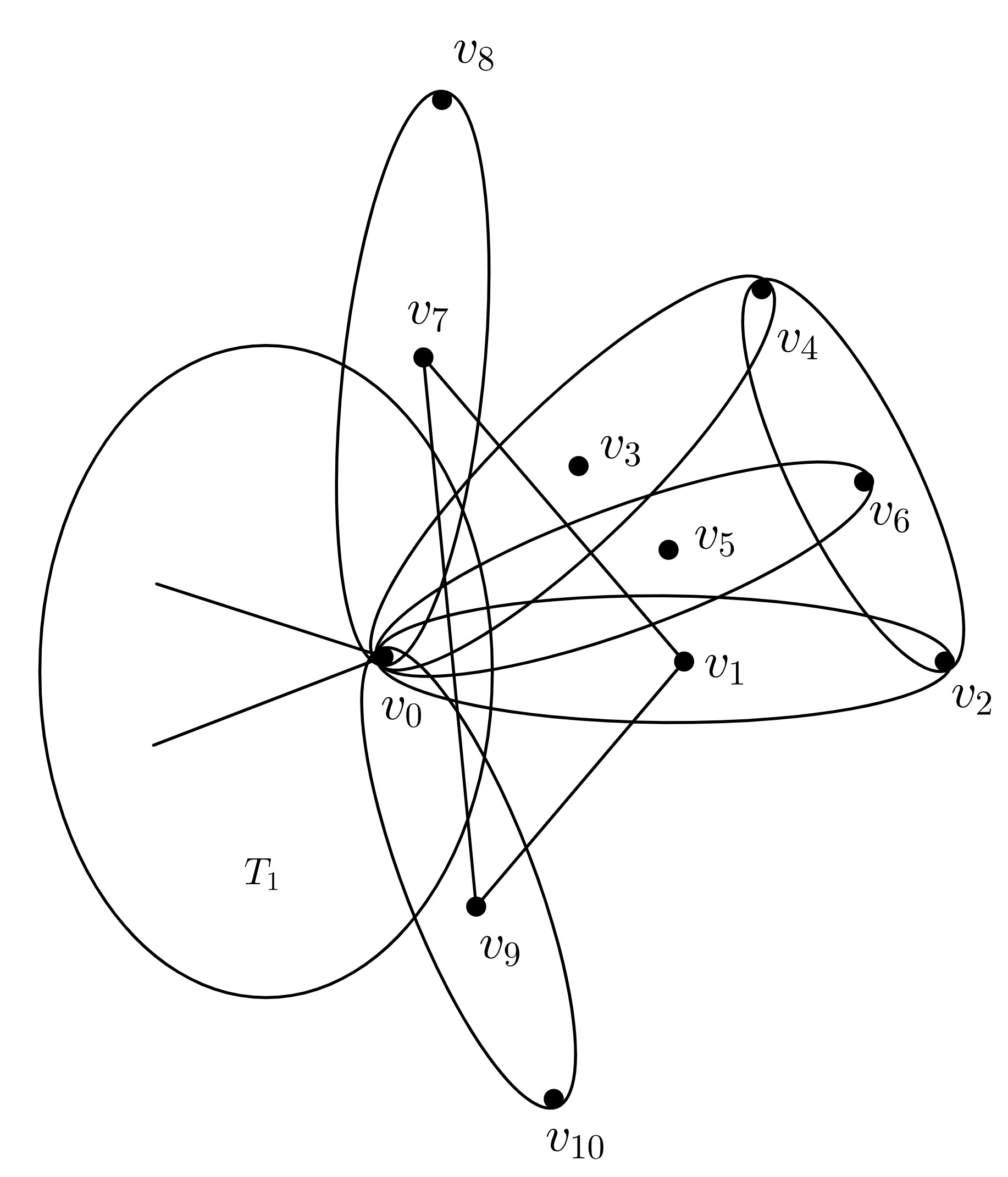}
		\caption*{$T_{11}$}
	\end{minipage}		
}
\caption{The hypergraphs $T_6$, $T_7$, $T_9$, and $T_{11}$}\label{T67911}
\end{figure}	 
When $n\equiv 6\pmod{18}$
, $T'_6$ is isomorphic to a  Berge-$C_3$. The only hyperedge that we can add is $\{v_1,v_3,v_5$\}, which will create a Berge-$C_4$. Hence, $T'_6$ is Berge-$C_4$-saturated, and so $T_6$ is Berge-$C_4$-saturated by Claim \ref{inducedgraph}.

When $n\equiv 7\pmod{18}$
, $T'_7$ is isomorphic to $T'$,  which is Berge-$C_4$-saturated. Therefore, $T_7$ is Berge-$C_4$-saturated by Claim \ref{inducedgraph}. 

When $n\equiv 9\pmod{18}$
, $T'_9$ is a subgraph of $T_1$, then it contains no copy of Berge-$C_4$. If we add a hyperedge in $T'_9$, then it contains at least two vertices of $T'$ who is Berge-$P_3$-connected, which will create a Berge-$C_4$. It indicates $T'_9$ is Berge-$C_4$-saturated, moreover $T_9$ is Berge-$C_4$-saturated by Claim \ref{inducedgraph}.

When $n\equiv 11\pmod{18}$
, $T'_{11}$ is a subgraph of $T_1$, then it contains no copy of Berge-$C_4$. There are a Berge-$P_3$ $(v_8,\{v_8,v_7,v_0\},v_0,\{v_0,v_2,v_1\},v_1,\{v_1,v_7,v_9\},v_9)$ between $v_8$ and $v_9$,  a Berge-$P_3$ $(v_8,\{v_8,v_0,v_7\},v_7,\{v_7,v_1,v_9\},v_9,\{v_9,v_0,v_{10}\},v_{10})$ between $v_8$ and $v_{10}$, and a Berge-$P_3$ $(v_7,\{v_1,v_7,v_9\},v_1,\{v_0,v_2,v_1\},v_0,\{v_0,v_{10},v_9\},v_{10})$ between $v_7$ and $v_{10}$. 
If we add a hyperedge in $T'_{11}$, then it contains at least two nonadjacent vertices of $T'$  or two vertices from $\{v_0,v_7,v_8\}$ and $\{v_0,v_9,v_{10}\}$, respectively. If the former holds, then there is a Berge-$C_4$  by $T'$ is Berge-$P_3$-connected. 
Since there are Berge-$P_3$s between $v_8$ and $v_9$, $v_8$ and $v_{10}$, $v_7$ and $v_{10}$, if the added hyperedge is the latter, we have a Berge-$C_4$.
Thus $T'_{11}$ is Berge-$C_4$-saturated, and $T_{11}$ is Berge-$C_4$-saturated  by Claim \ref{inducedgraph}.

\begin{figure}[H]
\centering{
	\begin{minipage}[t]{0.48\textwidth}
		\centering
		\includegraphics[scale=0.70]{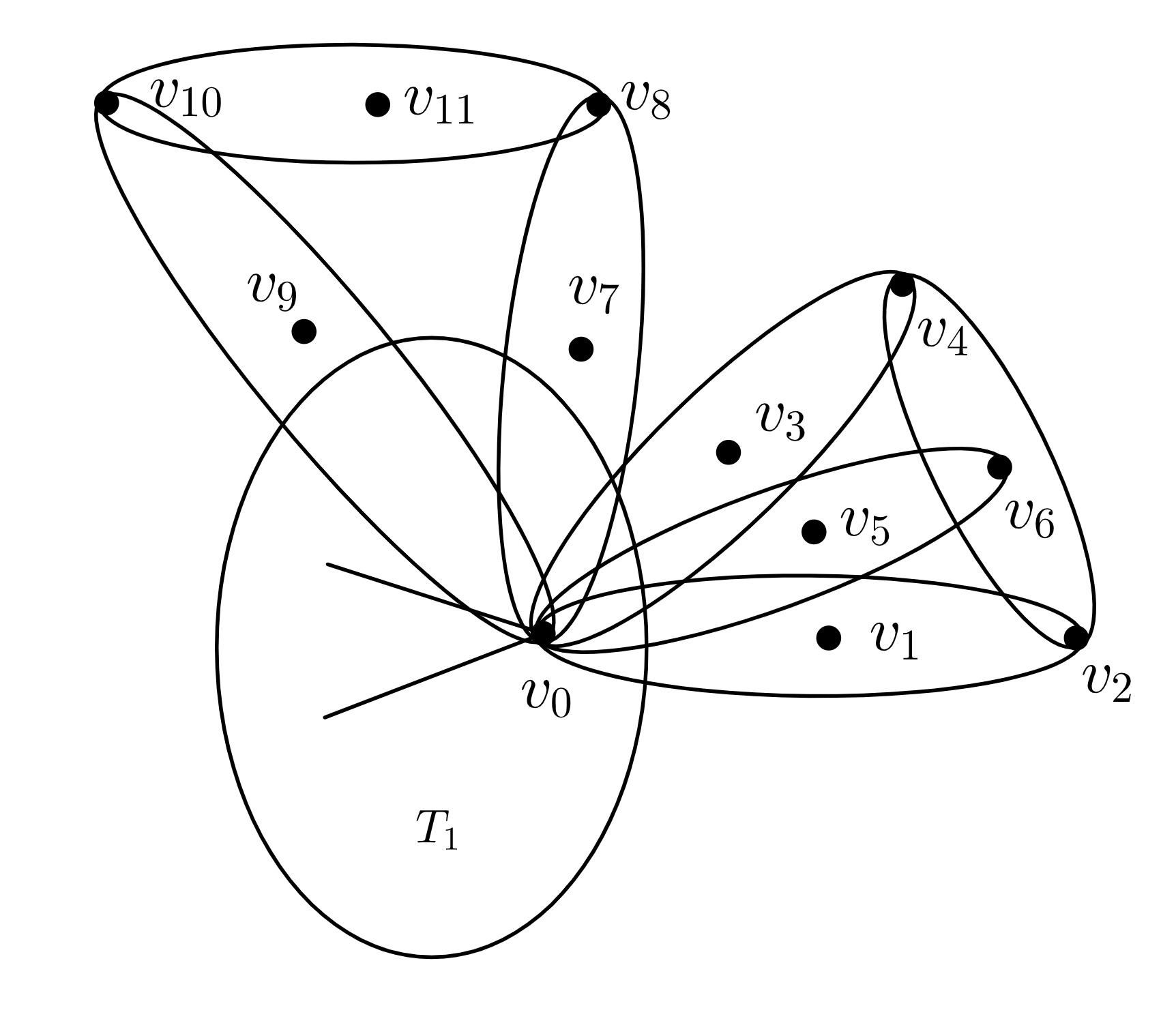}
		\caption*{$T_{12}$}
	\end{minipage}
	\begin{minipage}[t]{0.48\textwidth}
		\centering
		\includegraphics[scale=0.70]{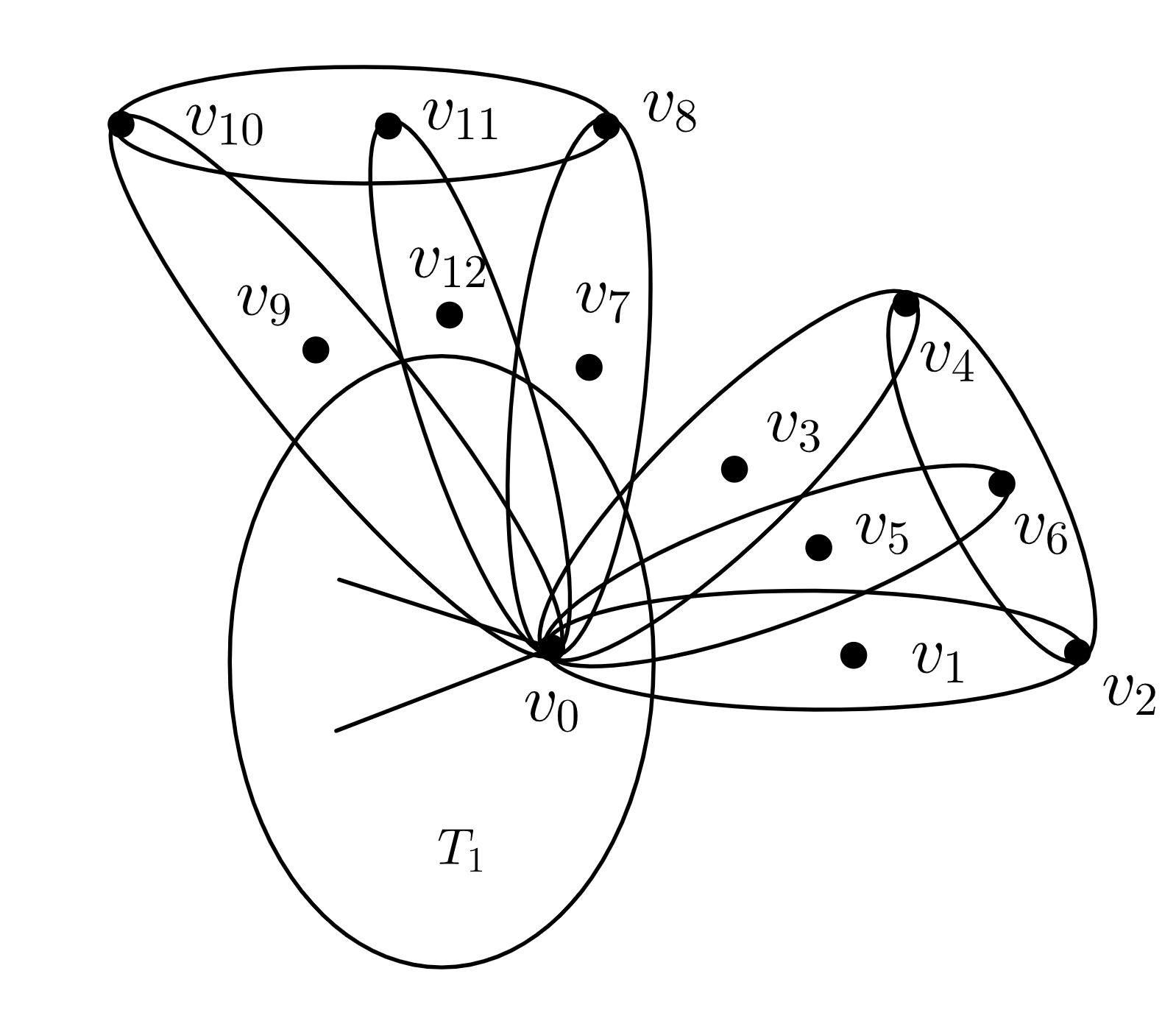}
		\caption*{$T_{13}$}
	\end{minipage}
	\begin{minipage}[t]{0.48\textwidth}
		\centering
		\includegraphics[scale=0.70]{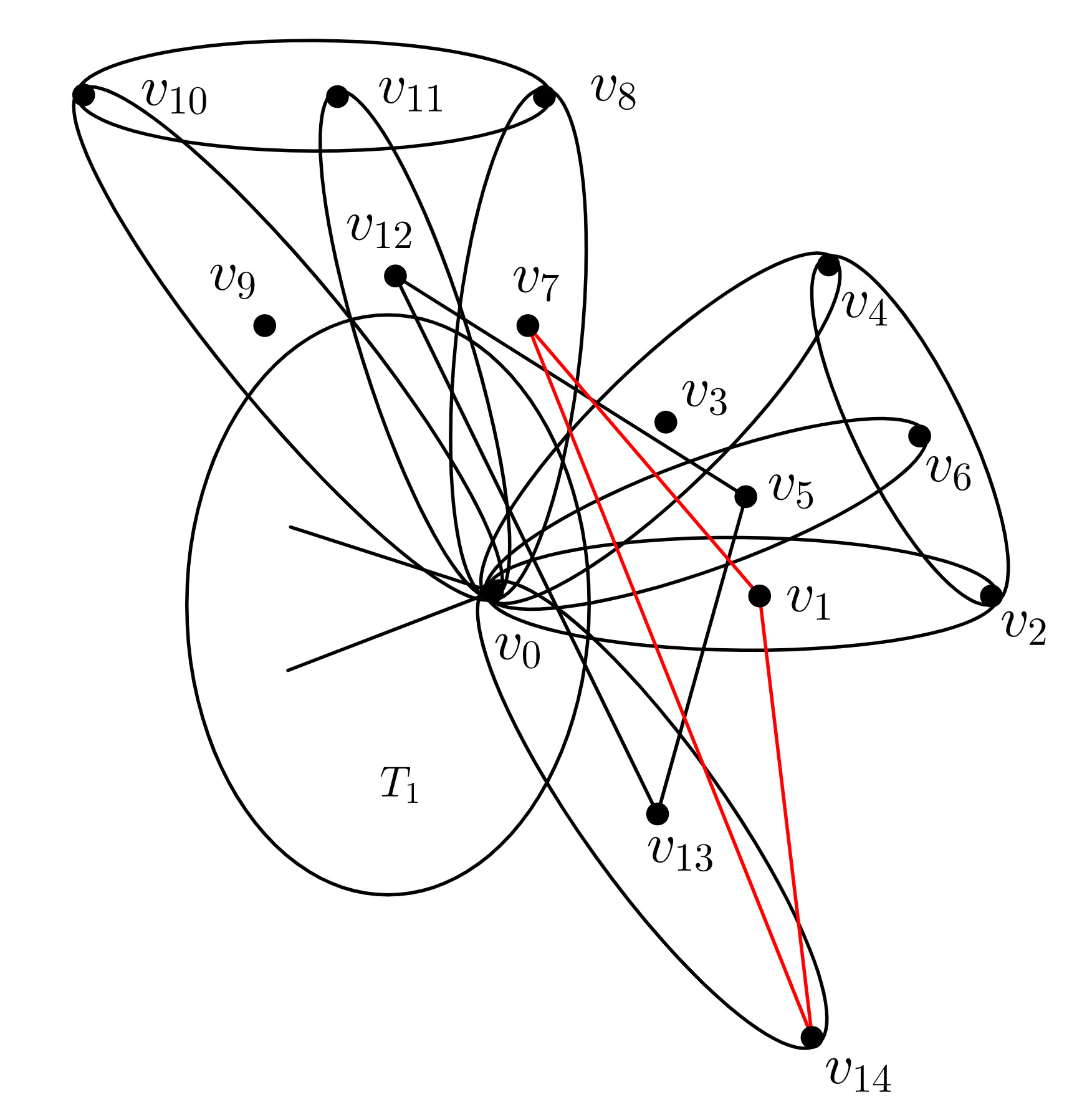}
		\caption*{$T_{15}$}
	\end{minipage}
	\begin{minipage}[t]{0.48\textwidth}
		\centering
		\includegraphics[scale=0.70]{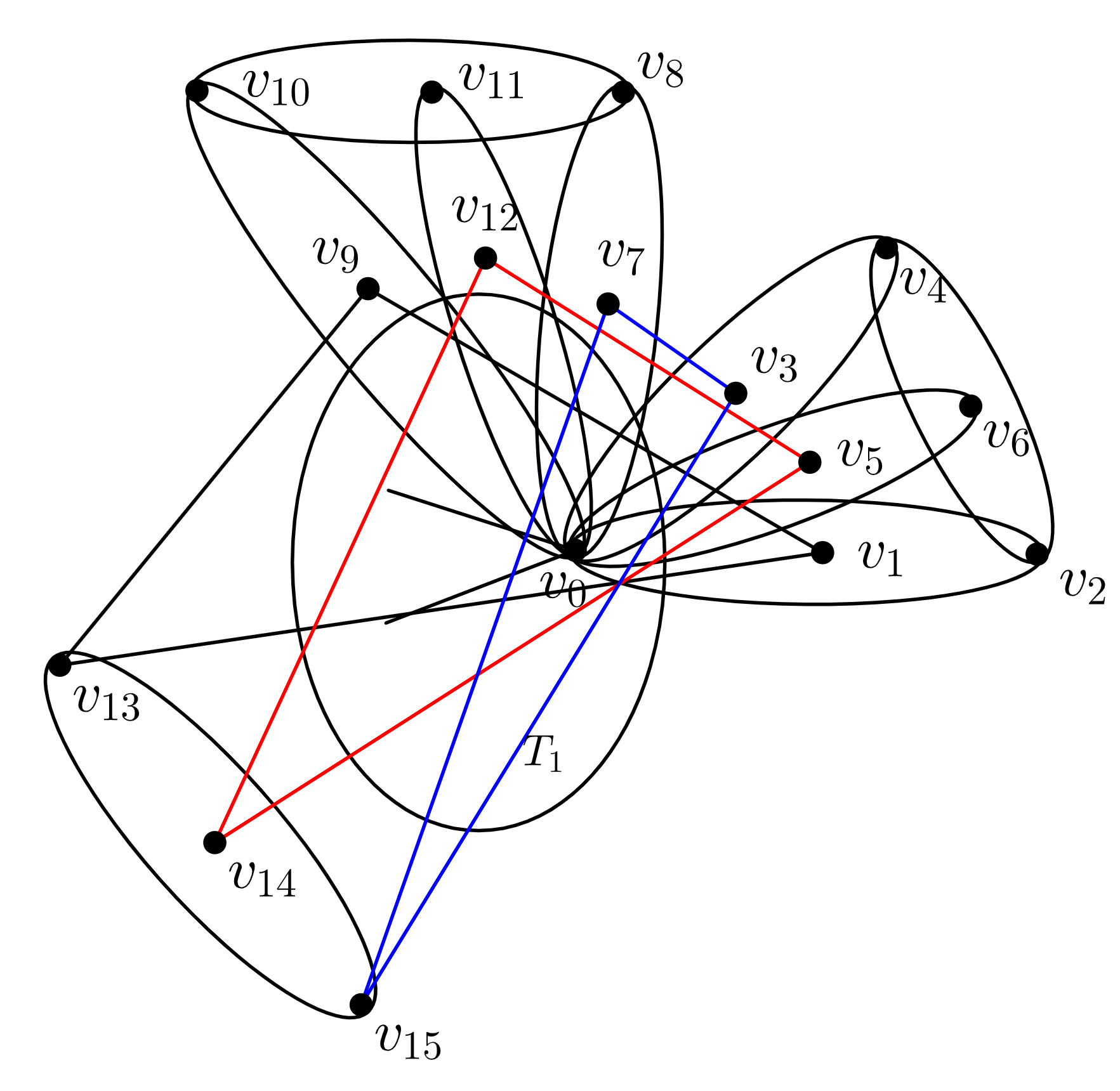}
		\caption*{$T_{16}$}
	\end{minipage}		
}
\caption{The hypergraphs $T_{12}$, $T_{13}$, $T_{15}$ and  $T_{16}$}\label{T12356}
\end{figure}	 

When $n\equiv 12\pmod{18}$, 
$T'_{12}$ is a subgraph of $T_1$, then it contains no Berge-$C_4$. 
If we add a hyperedge $e$ in $T'_{12}$ such that there is at least two vertices in $e$ from the vertex set $\{v_1,v_2,v_3,v_4,v_5,v_6\}$, we will have a Berge-$C_4$ by the hypergraph induced by the vertex set $\{v_0,v_1,v_2,v_3,v_4,v_5,v_6\}$ is isomorphic to $T'$, which is Berge-$P_3$-connected. If $e$ contains no vertex from vertex set $\{v_1,v_2,v_3,v_4,v_5,v_6\}$, note that the hypergraph induced by the vertex set $\{v_0,v_7,v_8,v_9,v_{10},v_{11}\}$ is isomorphic to a Berge-$C_3$, which is Berge-$C_4$-saturated, then it will create a Berge-$C_4$. It remains to consider the case that $e$ contains exactly one vertex from vertex set $\{v_1,v_2,v_3,v_4,v_5,v_6\}$. 
Observe that each vertex of $\{v_{10},v_{11},v_{8}\}$ has a Berge-$P_3$ to each vertex in $\{v_1,v_2,v_3,,v_4,v_5,v_6\}$.
If $e\cap \{v_{10},v_{11},v_{8}\}\ne \emptyset$, then there is a Berge-$C_4$. 
Thus we only need to consider $\{v_7,v_9\}\subseteq e$. 
But there exists a Berge-$P_3$ $(v_7,\{v_7,v_0,v_8\},v_8,\{v_8,v_{11},v_{10}\},v_{10},\{v_{10},v_0,v_9\},v_9)$ between $v_7$ and $v_9$. So any added hyperedge $e\in {{V(T'_{12})}\choose{3}}\setminus E(T'_{12})$ subject to linearity  will create a Berge-$C_4$. It implies that $T'_{12}$ is Berge-$C_4$-saturated and so  $T_{12}$ is Berge-$C_4$-saturated by Claim \ref{inducedgraph}.

When $n\equiv 13\pmod{18}$,
$T'_{13}$ is the hypergraph obtained from two copies of $T'$ by identifying $v_0$ as one, then it contains no copy of Berge-$C_4$.
For each hyperedge $e\in {{V(T'_{13})}\choose{3}}\setminus E(T'_{13})$ such that $E(T'_{13})\cup e$ is linear, note that $e$ contains at least two vertices in one copy of $T'$, which is Berge-$P_3$-connected.  If we add $e$, then it will create a Berge-$C_4$. Therefore $T'_{13}$ is Berge-$C_4$-saturated, further $T_{13}$ is Berge-$C_4$-saturated by Claim \ref{inducedgraph}.

When $n\equiv 15\pmod{18}$, let us show there is no Berge-$C_4$ in $T'_{15}$. Note that the hypergraph induced by the hyeredge set $E(T'_{15})\setminus\{v_1,v_7,v_{14}\}$ is a  subgraph of $T_1$. If there is a Berge-$C_4$, denoted by $\mathcal{C}_4$, then it contains hyperedge $\{v_1,v_7,v_{14}\}$. We first consider the case  $v_{14}\in supp( \mathcal{C}_4)$ as the first support vertex. The first and fourth hyperedges are $\{v_0,v_{13},v_{14}\}$ and $\{v_1,v_{7},v_{14}\}$, respectively. By the symmetry of $v_1$ and $v_7$, suppose the third hyperedge is $\{v_0,v_{1},v_{2}\}$. But there is no hyperedge connecting the first hyperedge and the third hyperedge except for the fourth hyperedge, a contradiction. 
Thus $v_{14}\notin supp( \mathcal{C}_4)$. Since $\{v_1,v_7,v_{14}\} \subseteq \mathcal{C}_4$ and $v_{14}\notin supp( \mathcal{C}_4)$, $\{v_1,v_{7}\}\subseteq supp( \mathcal{C}_4)$. 
Let $v_1$ be the first support vertex.  Then the first and fourth hyperedges are $\{v_1,v_{7},v_{14}\}$ and $\{v_0,v_{1},v_{2}\}$. The second hyperegde is  $\{v_0,v_{7},v_{8}\}$. But there is no hyperedge connecting the second hyperedge and the fourth hyperedge except for the first hyperedge, a contradiction. It implies that  there is no Berge-$C_4$ in $T'_{15}$. 
Since $T'_{13}$ is a subgraph of $T'_{15}$, it will create a Berge-$C_4$ to add a hyperedge, which does not contain $v_{13}$ or $v_{14}$.
Observe that each vertex in $T_{15}'$ is adjacent to $v_0$ and has a Berge-$P_2$ to $v_0$, then adding a hyperedge containing $v_{13}$ or $v_{14}$ will create a Berge-$C_4$.
Thus $T'_{15}$ is Berge-$C_4$-saturated, and so $T_{15}$ is Berge-$C_4$-saturated by Claim \ref{inducedgraph}.

When $n\equiv 16\pmod{18}$, suppose there is a Berge-$C_4$ in $T'_{16}$, denoted by $\mathcal{C}_4$. The hypergraph induced by $V(T'_{16})\setminus\{v_{13},v_{14},v_{15}\}$ is a subgraph of $T_1$, which contains no Berge-$C_4$. Thus $\big\{\{v_{13}, v_9,v_1\},\{v_{14}, v_{12},v_5\}, \{v_{15}, v_7,v_3\} \big\}\cap \mathcal{C}_4\ne \emptyset$. 
By the symmetry  of $v_{13},v_{14}$, and $v_{15}$,
we may assume that $\{v_{13}, v_9,v_1\} \subseteq \mathcal{C}_4$. 
If $v_{13}\in supp( \mathcal{C}_4)$ as the first support vertex, then the first and fourth hyperedges are $\{v_{13},v_{14},v_{15}\}$ and $\{v_1,v_{9},v_{13}\}$. By the  symmetry of $v_1$ and $v_9$, suppose the third hyperedge is $\{v_0,v_{1},v_{2}\}$ without loss of generality. But there is no hyperedge connecting the first hyperedge and the third hyperedge except for  the fourth hyperedge, a contradiction.
Then $\{v_1, v_9\}\subseteq supp( \mathcal{C}_4)$. Let $v_1$ and $v_9$ be the first and second support vertex, respectively. 
Then the second hyperedge and the fourth hyperedge are $\{v_0,v_9,v_{10}\}$ and $\{v_0,v_1,v_2\}$, but there is no hyperedge connecting them except for the first hyperedge, a contradiction. 
Hence, $T'_{16}$ contains no Berge-$C_4$. 
Since $T'_{13}$ is a subgraph of $T'_{16}$, it will create a Berge-$C_4$ to add a hyperedge, which does not contain $v_{13},v_{14},$ or $v_{15}$. 	Observe that $v_{13},v_{14},$ and $v_{15}$ have Berge-$P_2$s to $v_0$ and all other vertices are adjacent to $v_0$, then adding a hyperedge containing $v_{13},v_{14}$ or $v_{15}$ will create a Berge-$C_4$. We obtain $T'_{16}$ is Berge-$C_4$-saturated, moreover $T_{16}$ is Berge-$C_4$-saturated by Claim \ref{inducedgraph}. 

\begin{figure}[H]
\centering{
	\begin{minipage}[t]{0.48\textwidth}
		\centering
		\includegraphics[scale=0.70]{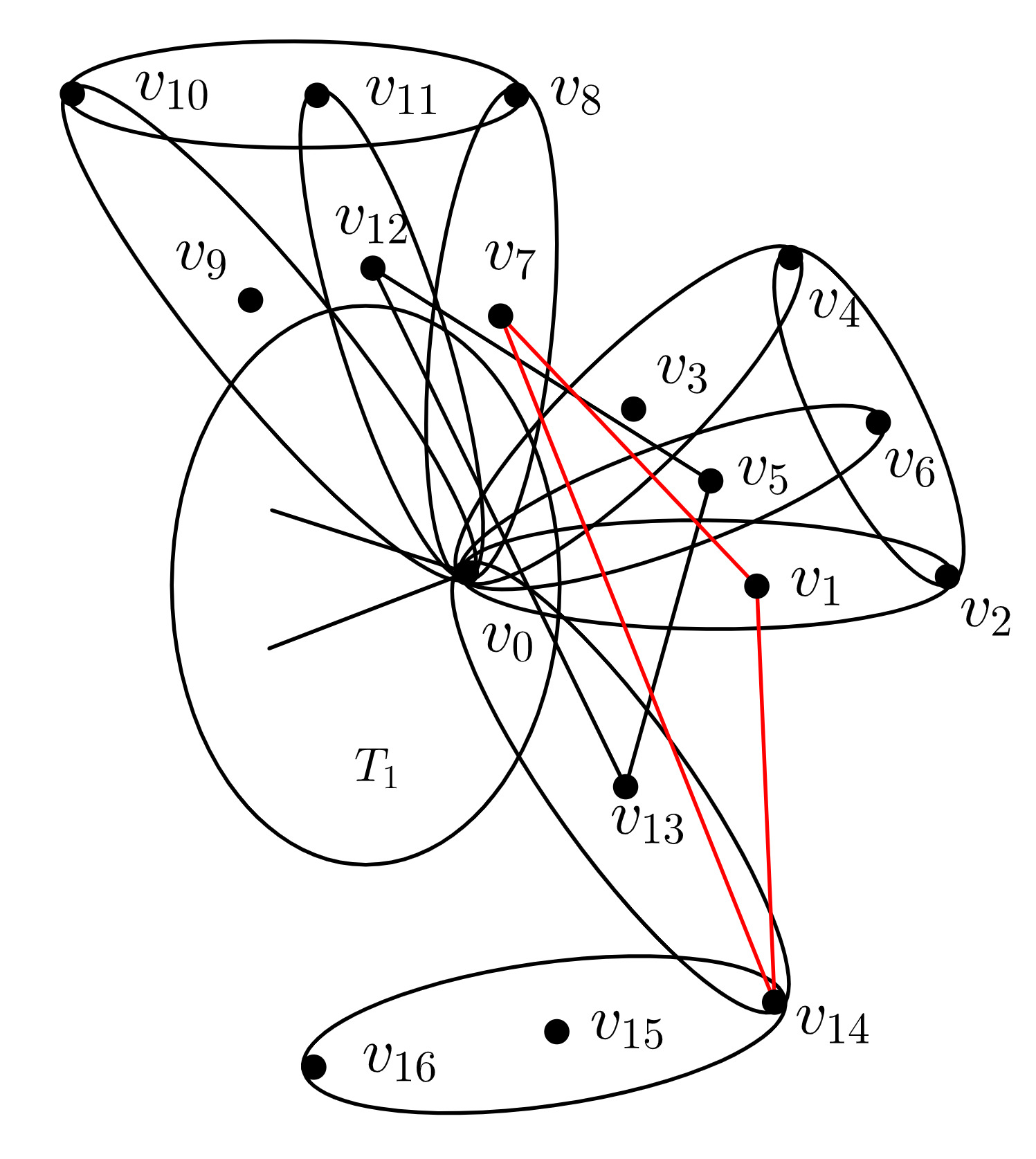}
		\caption*{$T_{17}$}
	\end{minipage}
	\begin{minipage}[t]{0.48\textwidth}
		\centering
		\includegraphics[scale=0.70]{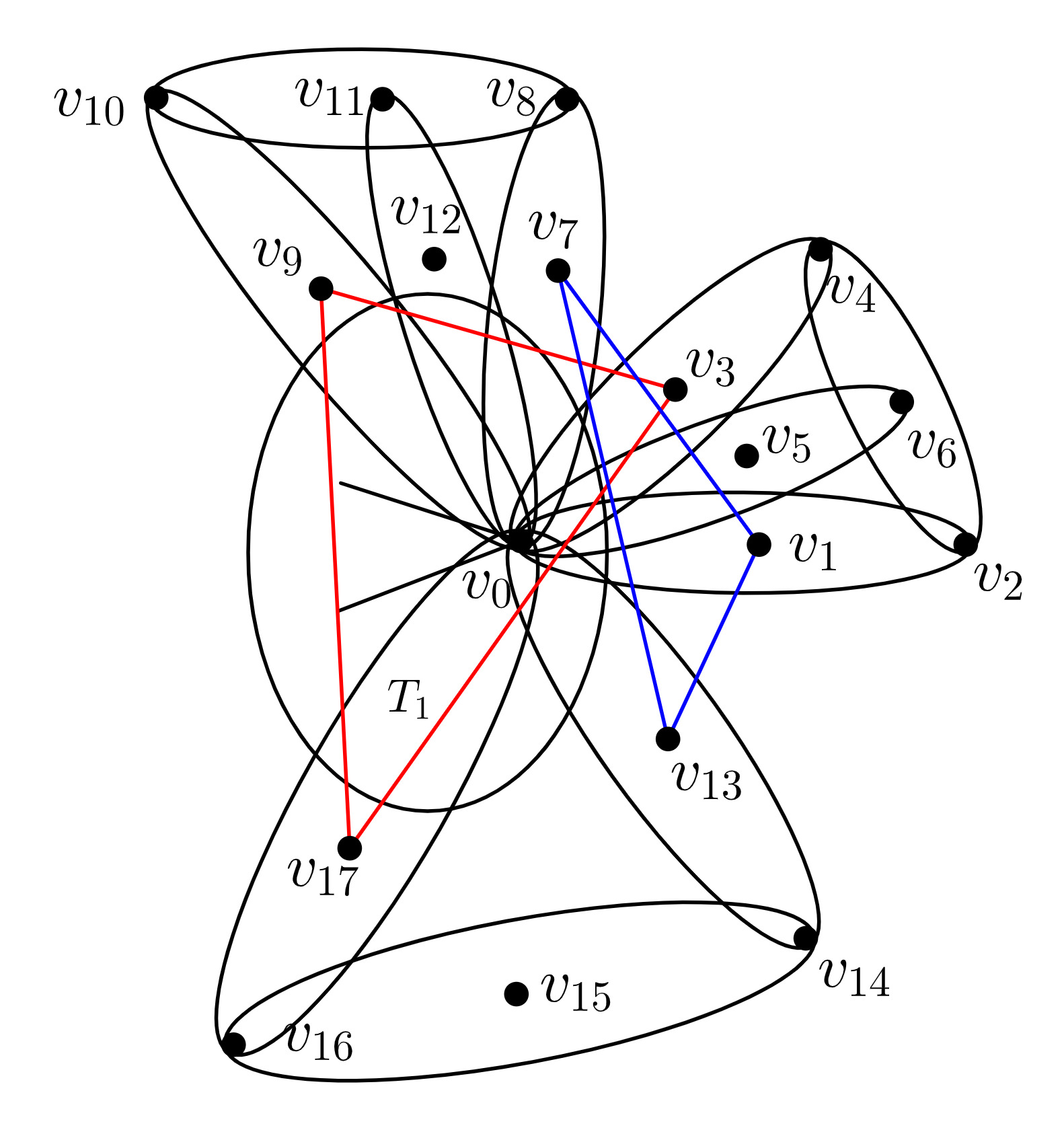}
		\caption*{$T_{18}$}
	\end{minipage}		
}
\caption{The hypergraphs $T_{17}$ and $T_{18}$}\label{T178}
\end{figure}

When $n\equiv 17\pmod{18}$, observe that $E(T'_{17})\setminus\{v_{14},v_{15},v_{16}\}$ is isomorphic to $T'_{15}$. 
If there is a Berge-$C_4$ in the hypergraph $T'_{17}$, then it contains hyperedge $\{v_{14},v_{15},v_{16}\}$, but the hyperedge is not contain in any Berge cycle, a contradiction. Thus $T'_{17}$ contains no copy of Berge-$C_4$. 
Adding a hyperedge containing no $v_{15}$ or $v_{16}$, then it will create a Berge-$C_4$ by $T'_{15}$ is Berge-$C_4$-saturated.
Note that both $v_{15}$ and $v_{16}$ have a Berge-$P_2$ to $v_0$, and any added hyperedge containing $v_{15}$ or $v_{16}$ will contain at least one vertex that is not contained by the Berge-$P_2$ and adjacent to $v_0$. Then the added hyperedge will create a Berge-$C_4$. Hence, $T'_{17}$ is Berge-$C_4$-saturated, furthermore, $T_{17}$ is Berge-$C_4$-saturated by Claim \ref{inducedgraph}.

When $n\equiv 0\pmod{18}$, 
$T'_{18}$ is a subgraph of $T_1$, then it contains no copy of Berge-$C_4$.
Next we consider that adding a hyperedge will create a Berge-$C_4$. 
Notice that the hypergraph induced by the vertex set $U'=\{v_0,v_{13},v_{14},v_{15},v_{16},v_{17}\}$ is isomorphic to a  Berge-$C_3$, which is Berge-$C_4$-saturated, and each vertex from $U'\setminus\{v_0\}$ has a Berge-$P_3$ to the vertices from $V(T'_{18})\setminus U$. Thus adding a hyperedge that contains at least one vertex from $U'$ will create a Berge-$C_4$. The hypergraph induced by $(V(T'_{18})\setminus U')\cup\{v_0\}$ is isomorphic to $T'_{13}$, which is Berge-$C_4$-saturated. 
It indicates $T'_{18}$ is Berge-$C_4$-saturated, and so $T_{18}$ is Berge-$C_4$-saturated by Claim \ref{inducedgraph}.

For $n\ge 1$, let $t=\big\lfloor\frac{n-1}{18}\big\rfloor$, and $i=n-18t$. By the definition of $i$, we have $i\in [18]$, and let 
$G=\big\lfloor\frac{n-1}{18}\big\rfloor T_1\cdot T_i'(v_0)$, where $T'_1=\{v_0\}$. 
Note that $T^*$ has 15 hyperegdes, then $e(G)=15t+e(T'_i).$ From the specific constructions of $T_i$, we have 
$$\operatorname e(T'_i)= \left\{
\begin{array}{lcl}
0, &&{i=1,2};\\
1, &&{i=3,4};\\
2, &&{i=5};\\
3, &&{i=6};\\
4, &&{i=7,8};\\
5, &&{i=9,10};\\
7, &&{i=11,12};\\
8, &&{i=13,14};\\
11, &&{i=15};\\
12,&&{i=16,17};\\
13, &&{i=18}.\\
\end{array}\right.$$
Since $t=\frac{n-i}{18}$, $e(G)=\frac{15(n-i)}{18}+e(T'_i)$. Then
$$sat_3^{lin}(n,\text{ Berge-}C_4)\le \operatorname e(G)= \left\{
\begin{array}{lcl}
\lfloor\frac{5(n-1)}{6}\rfloor, n\equiv1, 15, 16    &&{\pmod{18}};\\
\lfloor\frac{5(n-2)}{6}\rfloor, n\equiv 0, 2, 3, 4, 5, 6, 7, 17 &&{\pmod{18}};\\
\lfloor\frac{5(n-3)}{6}\rfloor, n\equiv8, 9, 10, 11, 12, 13    &&{\pmod{18}};\\
\lfloor\frac{5(n-4)}{6}\rfloor, n\equiv14    &&{\pmod{18}}.\\
\end{array}\right.
$$

$\hfill\qedsymbol$

\section{The proof of Theorem \ref{lwc4}}\label{s5}
For a linear 3-uniform hypergraph $H$, we construct an auxiliary graph $H_0$ with $V(H_0)=V(H)$ and $E(H_0)=\{v_1v_2: v_1,v_2\in e$ for some $e\in E(H)$\}. For every hyperedge $e=\{v_1, v_2, v_3\}\in E(H)$, there is a triangle $v_1v_2v_3$ in $H_0$, like Figure \ref{edge}.
\begin{figure}[H]
\centering
\includegraphics[scale=0.70]{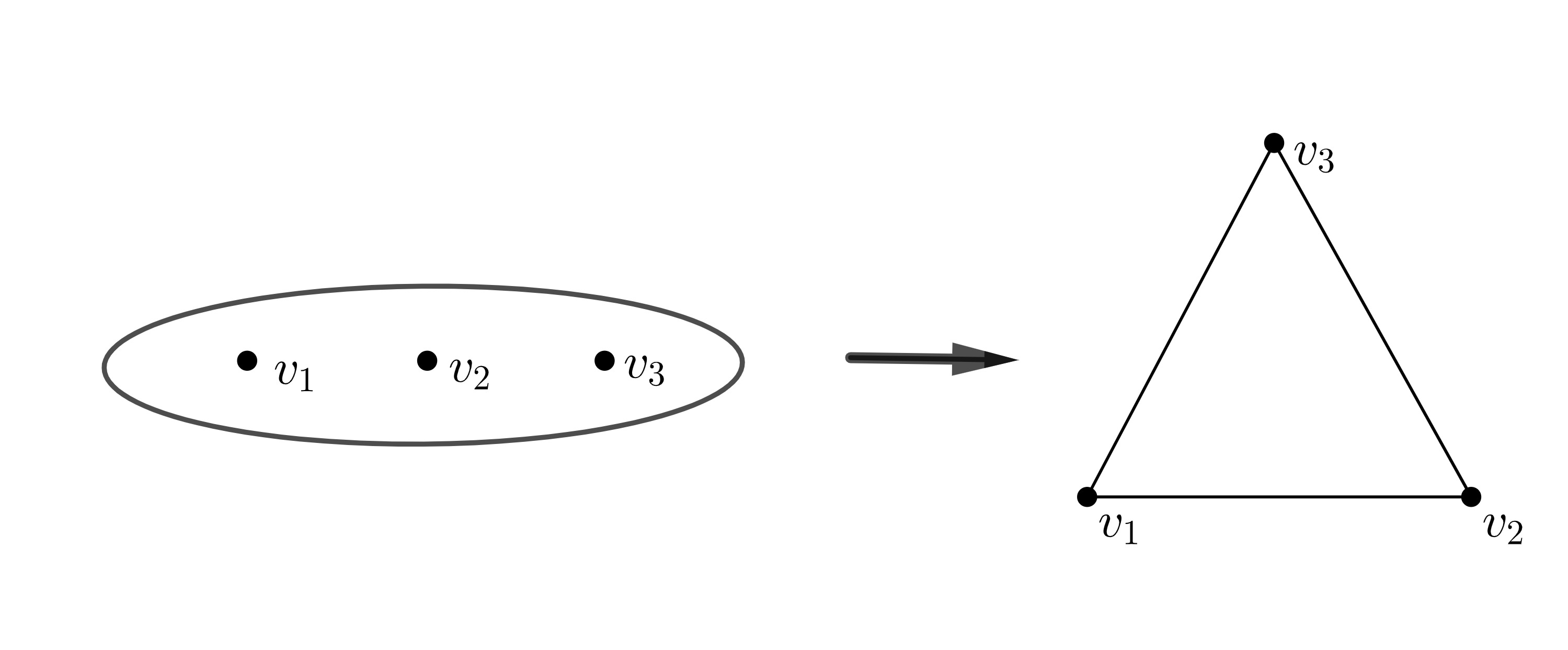}
\caption{The corresponding relation between $H$ and $H_0$}\label{edge}

\end{figure}

Since $H$ is a linear $3$-uniform hypergraph, 
we have $e(H_0)=3e(H)$ and  $d_{H_0}(v)=2d_H(v)$ for each vertex $v\in V(H)$. We can estimate the number of edges in $H_0$ to estimate the number of hyperedges in $H$.

Since $H$ is Berge-$C_4$-saturated and disconnected, each component of $H$ is Berge-$P_3$-connected. 
For any  component $H'$ of  $H$ with $n'$ vertices and $m'$ hyperedges, then the auxiliary graph $H'_0$ has $n'$ vertices and $3m'$ edges. 
Let $v_0$ be a vertex  with minimum degree   in $H'_0$. 
Since $H'$ is Berge-$P_3$-connected, $H'$ has diameter at most 3. Let $X$, $Y$, and $Z$ be the set of vertices that are distance 1, 2, and 3, respectively, from $v_0$.  Since the vertex sets $X,Y,Z$ are divided by distance, for any hyperedge $e\in E(H')$, there exist at least two vertices belong to the same class set among three vertices in $e$. We choose exactly one pair vertices in $e$ that belong to the same class set and remove the edge incident to it in $H'_0$.  Denote by $R$ the edge set of 
all removed edges in $H'_0$. 
Thus  $R\subseteq E(H_0'[X])\cup E(H_0'[Y]) \cup E(H_0'[Z]) $ and the number of removed edges is $|R|=m'$ by $H'$ is linear. 

After removing, the remained graph $H'_1=H_0'-R$ is still connected. Let $T$ be the breadth-first tree of ${H'_1}$ rooted at  $v_0$. Since the vertex set $X,Y,Z$ are divided by distance from $v_0$,   $T\subseteq E(H_1'[v_0,X])\cup E(H_1'[X,Y])\cup E(H_1'[Y,Z])$. We can see $R\cap E(T)=\emptyset$. 
We estimate the edges that are not covered by $R$ or $T$. Let ${H'_2}=H'_0-R-T$. 
When $\delta_{H'_0}\ge 6$, we have $3e(H')=e(H'_0)\geq \frac{n' \cdot \delta_{H'_0}}{2}\ge 3n'$. Next we consider the lower bound of the number of hyperedges in $H'$ when $\delta_{H'_0}=0, 2$, or $4$. 
When $\delta_{H'_0}=0$, $H_0$ is an isolated vertex. 
Therefore, we discuss $\delta_{H'_0}=2$ or $ 4$ in the proof. 
We will complete our proof by the following two cases. 

\medskip

\noindent{\bf Case 1. $\delta_{H'_0}=4$. }

\medskip

Define 
$$Y_1=\{y\in Y: N_{H'_2}(y)\cap X \ne \emptyset\}, ~
Y_2=\{y\in Y\setminus Y_1: N_{H'_2}(y)\cap  Y\ne \emptyset\},$$ 
$$Y_3=\{y\in Y\setminus(Y_1\cup Y_2): |N_{H'_0}(y)\cap Z|=1\},~\text{and }
Y_4=\{y\in Y\setminus(Y_1\cup Y_2): |N_{H'_0}(y)\cap Z|\ge 2\}.$$ 
Thus we have 
\begin{align}\label{e(XY)}
e(H_2'[X,Y])= e(H_2'[X,Y_1])\ge |Y_1|.
\end{align}
By the definition of $H'_2$, if $y\in Y_2$, then $y$ belongs to some hyperedge of the form $YYY$ and so there is a path of length 2 containing $y$ in $H'_2[Y_2]$. Therefore, for any component of $H'_2[Y_2]$, there is a path of length 2, that is any component of $H'_2[Y_2]$ has at least three vertices.  Assume that there are $\ell $ components in $H'_2[Y_2]$. Then $|Y_2|\ge 3\ell $ and 
\begin{align}\label{e(Y)}
e(H'_2[Y])\ge e(H'_2[Y_2])\ge |Y_2|-\ell \ge \frac{2|Y_2|}{3}.
\end{align}
By $\delta_{H_0'}=4$, that is $\delta_{H'}=2$, for any vertex $y\in Y\setminus(Y_1\cup Y_2)$, there is at least one hyperedge containing $y$ and a vertex in $Z$ and so  $N_{H_0'}(y)\cap Z\ne \emptyset$. Thus $Y=Y_1\cup Y_2\cup Y_3\cup Y_4$ and   for any $y\in Y_3$, there are  exactly two hyperedges of the  form $XYY$ and $YYZ$ containing $y$, respectively. Define  $Z_i=\{v\in Z: |N_{H'_0}(v)\cap Y|=i\}$ for $1\le i\le |Y|$.  We can see $|Z|=\sum_{1\le i \le |Y|}|Z_i|$. 
Let $$Y_{31}=\{y\in Y_3: N(y)\cap Z\subseteq Z_2\} \text{, } Y_{32}=\{y\in Y_3: N(y)\cap Z\subseteq  \bigcup_{3\le i\le |Y|}Z_i\},$$ and $$Z_{21}=\{z\in Z_2: z\in  N_{H'_0}(y) \text{ for some } y\in Y_{31}\}.$$
Note that for any vertex $y\in Y_3$, there exists a hyperedge  of the form $YYZ$ containing $y$. We have  $z\in \bigcup_{2\le i\le |Y|}Z_i$ for any $z\in N_{H'_0}(y)\cap Z$, where $y\in Y_3$. 
It follows that
$Y_{3}=Y_{31}\cup Y_{32}$. 
Observe that vertices in $Z$ are leaves of $T$, then we have  \begin{align}\label{e(YZ)}
e(H_2'[Y,Z])=e(H_0'[Y,Z])-|Z|=\sum_{2\le i\le |Y|}|Z_i|(i-1).
\end{align}
Notice that  $\delta_{H_0'}=4$ and $\delta_{H'}=2$, for any vertex $z\in Z_1 \cup Z_{21}$, we have that there exists at least one hyperedge of the form $ZZZ$ containing $z$ and so there exists a path of length 2 containing $z$ in $H_2'[Z]$. Assume that there are $t$ components containing at least one vertex in $Z_1\cup Z_{21}$ in $H_2'[Z]$. Denote by $Z'$ the union of these components. Then $|Z'|\ge |Z_1\cup Z_{21}|$, $|Z'|\ge 3t$ and 
\begin{align}\label{e(ZZ)}
e(H_2'[Z])\ge |Z'|-t\ge \frac{2|Z'|}{3} \ge \frac{2(|Z_1|+|Z_{21}|)}{3}.
\end{align}
By definitions of $Y_{31}$, $Z_{21}$,   $Y_{32}$, and $Y_4$,   we have $|Y_{31}|\le 2|Z_{21}|$, 
$|Y_{32}|\le \sum_{3\le i\le |Y|}|Z_i|\cdot i$, and $e(H_0'[Y,Z])=\sum_{1\le i\le |Y|}|Z_i|\cdot i\ge 2|Y_4|+|Y_{31}|+|Y_{32}|$. Then 
\begin{align}\label{Y3Y4}
|Y_3|+|Y_4|&\notag
=|Y_{31}|+|Y_{32}|+|Y_4|\\
&\notag\le \frac{1}{2}\big\{\sum_{1\le i\le |Y|}|Z_i|\cdot i+2|Z_{21}|+\sum_{3\le i\le |Y|}|Z_i|\cdot i\big\}\\
&=\sum_{3\le i\le |Y|}|Z_i|\cdot i+|Z_2|+|Z_{21}|+\frac{|Z_1|}{2}.
\end{align}
Combining  Inequations (\ref{e(XY)}), (\ref{e(Y)}), (\ref{e(YZ)}), (\ref{e(ZZ)}), and (\ref{Y3Y4}), we have 
\begin{align*}
e(H_2')&\ge  e(H_2'[X,Y])+e(H_2'[Y])+ e(H_2'[Y,Z])+e(H_2'[Z])\\
&\ge |Y_1|+\frac{2|Y_2|}{3}+\sum_{2\le i\le |Y|}|Z_i|(i-1)+\frac{2(|Z_1|+|Z_{21}|)}{3}\\
&\ge  |Y_1|+\frac{2|Y_2|}{3}+\frac{4}{9}\big(\sum_{3\le i\le |Y|}|Z_i|\cdot i+|Z_2|+|Z_{21}|+\frac{|Z_1|}{2}+\sum_{1\le i\le |Y|}|Z_i|\big)\\
&\ge \frac{4}{9}(|Y_1|+|Y_2|+|Y_3|+|Y_4|+|Z|)\\
&=\frac{4}{9}(|Y|+|Z|).
\end{align*}
Therefore, \begin{align*}
3m'&=3e(H')=e(H_0')=e(R)+e(T)+e(H_2')\\
&\ge m'+n'-1+\frac{4}{9}(|Y|+|Z|)\\
&=m'+n'-1+\frac{4}{9}(n'-5),
\end{align*}
and 
$$e(H')=m'\geq\frac{13n'-29}{18}.$$

\medskip 

\noindent{\bf Case 2. $\delta_{H_0'}=2$. }

\medskip

Let $v_0$ be a vertex of degree 2 and $N_{H_0'}(v_0)=\{x_1, x_2\}$. 
Define $$Y_1=\{y\in Y: \text{ there is no hyperedge of the form $YYY$ or $YYZ$ containing $y$} \}, $$
$$Y_2=\{y\in Y\setminus Y_1: \text{ there is a hyperedge of the form $YYY$ containing $y$}\},\text{ and  ~}$$  $$Y_3=\{y\in Y\setminus (Y_1\cup Y_2): \text{  there is a hyperedge of the form $YYZ$ containing $y$}\}.$$ 
Thus $Y=Y_1\cup Y_2\cup Y_3$. 
For any $y\in Y$, 
since $H'$ is Berge-$P_3$-connected, there is a Berge-$P_3$ between $y$ and $v_0$, saying $(y,e_1,v_2,e_2,v_3,e_3,v_0)$, where $y,v_2,v_3,v_0$ are distinct support vertices and $e_1, e_2, e_3$ are distinct hyperedges. 
By $\delta_{H'}=1$, $e_3=\{v_0,x_1, x_2\}$, that is $v_3\in \{x_1,x_2\}$, which implies that $v_2\in Y$. 
Thus $e_1$ is of the form $YYY$, $YYZ$, or $XYY$. 
For any $y\in Y_1$, $e_1$ is of the form $XYY$, 
saying $e_1=\{x_i,y,v_2\}$, then $v_3=x_{3-i}$, else $v_3=x_i$, we have $\{v_2,x_i\}\subseteq e_1\cap e_2$, a contradiction. 
Thus  $\{x_1,x_2\}\subseteq N_{H'_0}(v_2)$. By $H$ is linear, for any $v_2$ there are at most two neighbours of it  in $Y_1$. Let $Y'=\{y\in Y: \{x_1,x_2\}\subseteq N_{H_0'}(y)\}$. Then $|Y_1|\le 2|Y'|$ and 
\begin{align}\label{2eXY}
e(H_2'[X,Y])\ge  e(H_2'[X,Y'])=|Y'|\ge \frac{|Y_1|+|Y'|}{3}\ge \frac{|Y_1|}{3}.
\end{align}
For any $y\in Y_2$, there is a path of length 2 containing $y$ in $H_2'[Y_2]$ and so  each component containing $y$  in $H_2'[Y_2]$ has at least three vertices. Assume that there are $\ell$ components in $H_2'[Y_2]$. Then $|Y_2|\ge 3\ell$ and 
\begin{align}\label{2eYY}
e(H_2'[Y])\ge e(H_2'[Y_2]) \ge |Y_2|-\ell\ge \frac{2|Y_2|}{3}.
\end{align}
Let 
$$Z_1=\{z\in Z: d_{H_0'}(z)=2 \text{ and } N_{H_0'}(z)\cap Z \ne \emptyset\}, ~
Z_{11}=\{z\in Z_1: N_{H_0'}(z)\cap Z_1\ne \emptyset\},$$
$$Z_2=\{z\in Z: \text{there are at least two hyperedges of the form $YZZ$ containing $z$}\}, $$
$$Z_{12}=\bigcup_{z\in Z_2} (N_{H_0'}(z)\cap Z_1),  \text{ and } Z_{2i}=\{z\in Z_2:  |N_{H_0'}(z)\cap Z_1|=i\}
\text{ for } 0\le i \le |Z_{1}|. $$
Then $|Z_2|=\sum_{0\le i\le |Z_1|}|Z_{2i}|$. By definition of $Z$, for any $z\in Z_1$, there is exactly one hyperedge of the form $YZZ$ containing $z$.
From the definition of $Z_{11}$, the vertices in $Z_{11}$ always appear in pairs. We have the following claim. 
\begin{claim}\label{Z22}
\rm 	$|Z_{11}|\le 6$.
\end{claim}
\noindent{\textit{Proof.}}   Suppose to the contrary that $|Z_{11}|\ge 8$, and let the hyperedges containing them  be $e_1=\{y_1, z_{11},z_{12}\}, e_2=\{y_2, z_{21},z_{22}\}, e_3=\{y_3, z_{31},z_{32}\}, \text{  and } e_4=\{y_4, z_{41},z_{42}\}$, where $z_{i'j'}\in Z_{11}$ for any $i',j'\in [2]$. Since there is a Berge-$P_3$ between the nonadjacent vertices with degree 2 in $Z$, saying $z_{i1}$ and $z_{j1}$,  there exists a hyperedge containing $y_i$ and $y_j$ for any $i,j\in [4]$ and $i\ne j$. 
Let $e_{ij}$ be the hyperedge containing $y_i$ and $y_j$. 
If there is no hyperedge containing three vertices of $y_1,y_2,y_3,y_4$, that is $e_{ij}\ne e_{tp}$ for any $i,j,t,p\in [4]$ with $\{i,j\}\ne \{t,p\}$. It can be seen $(y_1,e_{12},y_2,e_{23},y_3,e_{34},y_4,e_{14})$ is a Berge-$C_4$ in $H'$, a contradiction. 
Next we assume that  there exists a hyperedge containing three vertices of $y_1,y_2,y_3,y_4$, saying $e_{12}=e_{23}=e_{13}$, then $\{y_1,y_2,y_3\} \subseteq N_{H_0'}(x_1)\setminus N_{H_0'}(x_2)$ or $\{y_1,y_2,y_3\} \subseteq N_{H_0'}(x_2)\setminus N_{H_0'}(x_1)$. 
By contradiction, suppose $y_1\in N_{H_0'}(x_1)$ and $y_2\in N_{H_0'}(x_2)$. Let $e'_i$ be the hyperedge containing $y_i$ and $x_i$ for $i\in [2]$. We can see $(x_1,\{v_0,x_1,x_2\},x_2,e'_2,y_2,e_{12},y_1,e'_1)$ is a Berge-$C_4$ in $H'$, a contradiction. Without loss of generality, we can assume  $\{y_1,y_2,y_3\} \subseteq N_{H_0'}(x_1)\setminus N_{H_0'}(x_2)$. Let the hyperedge containing $\{x_1, y_i\}$ be $e^{i}$ for any $i\in [3]$. Since $e_{12}=e_{23}=e_{13}=\{y_1,y_2,y_3\}$ and $H'$ is linear, $e^{i}\ne e^{j}$ for any $i,j\in [3]$ and$i\ne j$. 
If $e_{i4}=e^i$ for some $i\in [3]$, saying $e_{34}=e^3$, that is  $e_{34}=\{x_1,y_3,y_4\}$, then $e_{j4}\ne e^j$ for any $j\in [2]$, else $e_{j4}=\{x_1,y_j,y_4\}$, we have $e_{j4}\cap e_{34}=\{x_1, y_4\}$, which contradicts the linearity of $H'$.
Thus we can always assume $e^j\ne e_{j4}$ for any $j\in [2]$.  
We can see $(x_1,e^1,y_1,e_{14},y_4,e_{24},y_2,e^2)$ is a Berge-$C_4$, a contradiction. 
Thus $|Z_{11}|\le 6$. 
$\hfill\qedsymbol$

\begin{claim}\label{Z12}
For any $z\in Z$ with  $|N_{H_0'}(z)\cap Z_1|\ge 2$, then $z\in Z_2$.
\end{claim}
\noindent{\textit{Proof.}}
For any $z\in Z$ with  $|N_{H_0'}(z)\cap Z_1|\ge 2$, saying $\{z_1,z_2\}\subseteq N_{H_0'}(z)\cap Z_1$,  by the definition of $Z_1$, then there are at least two hyperedges of the form $YZZ$ containing $\{z,z_1\}$ and $\{z,z_2\}$, respectively. It implies that $z\in Z_2$. 
$\hfill\qedsymbol$

\medskip 

We choose an edge set $F_1$ from $E(H_2'[Y,Z_2])$. Let $F_1=\{yz \in E(H_2'[Y,Z_2]): $ there exists some vertex $z_1\in Z$ such that $\{y,z,z_1\}$ is a hyperedge in $H'$\}. 
Then  $|Z_{12}|=  \sum_{1\le i \le |Z_{1}|}i\cdot|Z_{2i}|$ and 
\begin{align}\label{eF2}
|F_1|&\notag \ge  \sum_{2\le i \le |Z_{1}|}(i-1)|Z_{2i}|+|Z_{21}|+|Z_{20}|\\
&\notag \ge \frac{1}{3}( \sum_{1\le i \le |Z_{1}|}i\cdot|Z_{2i}|+ \sum_{0\le i \le |Z_{1}|}|Z_{2i}|)\\
&\ge \frac{1}{3}( |Z_{12}|+|Z_2|).
\end{align}
Let $$Z_{3}=\{z\in Z\setminus (Z_1\cup Z_2): \text{there exists at least one hyperedge of the form $YYZ$ containing } z\},$$ $$Z_{13}=\cup_{z\in Z_3}(N_{H_0'}(z)\cap Z_1),~
Z_{31}=\{z\in Z_3: N_{H_0'}(z)\cap Z_1 \ne \emptyset\},~Z_{32}=Z_3\setminus Z_{31},$$
$$Z_{3i}^{j}=\{z\in Z_{3i}: \text{ there are $j$ hyperedges of the form $YYZ$ containing } z\}$$ for $i\in [2]$ and $1\le j \le |Y|^2.$
Then $|Y_3|\le \sum_{1\le j\le |Y|^2} (|Z^j_{31}|+|Z^j_{32}|)\cdot 2j$.  For any vertex $z_0\in Z_{13}$, we have $z_0\notin Z_{12}$ by definition. 
For any vertex $z'\in N_{H_0'}(z_0)\cap Z$,  by Claim \ref{Z12}, then $|N_{H_0'}(z')\cap Z_{13}|=1$
and so $|Z_{13}|=|Z_{31}|=\sum_{1\le j\le |Y|^2}|Z_{31}^j|$.
Now we choose an edge set from $E(H_2'[Y,Z_3])$, denoted by $F_2$. Let $F_2=\{yz \in e(H_2'[Y,Z_3]): \{z,y,y_1\}$ or $\{y,z,z_1\}$ is a hyperedge, where $y,y_1\in Y$, $z\in Z_3$ and $z_1\in Z_{13}$\}.
Thus 
\begin{align}\label{eF}
|F_2|& \notag=\sum_{1\le j\le |Y|^2}(2j+1-1)|Z_{31}^j|+\sum_{1\le j\le |Y|^2}(2j-1)|Z_{32}^j|
\\
&
\notag\ge \frac{1}{3}\big\{\sum_{1\le j\le |Y|^2}2j\cdot( |Z_{31}^j|+ |Z_{32}^j|)+\sum_{1\le j\le |Y|^2}(|Z_{31}^j|+|Z_{32}^j|)+\sum_{1\le j\le |Y|^2}|Z_{31}^j|\big\}\\
&\ge \frac{1}{3}(|Y_3|+|Z_{3}|+|Z_{13}|).
\end{align}

Next we consider the vertices in $Z_{14}=Z_1\setminus (Z_{11}\cup Z_{12}\cup Z_{13})$. For any  $z\in Z_{14}$, saying $Z\cap N_{H_0'}(z)=\{z'\}$, then $d_{H_0'}(z')\ge 4$, one hyperedge containing $z'$ is of the form $YZZ$ which contains $z$, 
the second hyperedge containing $z'$  is of the form $YYZ$, $YZZ$, or $ZZZ$. By definition of $Z_{14}$,  we have the second hyperedge containing $z'$ is of the form $ZZZ$. Let $$Z_4=   Z\setminus(Z_1\cup Z_2\cup Z_3),$$  and  $$Z'_4= \{z\in Z: \text{ there is at least one hyperedge of the form $ZZZ$ containing $z$}\}.$$ 
Then $N_{H_0'}(z_1)\cap Z\subseteq Z_4$  and  $N_{H_0'}(z_1)\ne N_{H_0'}(z_2)$ for any $z_1,z_2\in Z_{14}$, which implies that
$|Z_{14}|\le |Z_{4}|$. 
For any $z\in Z_4$, since $z\notin Z_1\cup Z_2\cup Z_3$, there is at least one hyperedge of the form $ZZZ$ containing $z$, that is, $Z_4\subset Z'_4$. For each vertex in $H_2'[Z'_4]$, there is a path of length 2 containing it in $H_2'[Z'_4]$ and so there is a path of length 2 in each component of $H_2'[Z'_4]$. Assume that there $t$ components in $H_2'[Z_4]$. Then $|Z'_4|\ge 3t$ and 
\begin{align}\label{eZ}
e(H_2'[Z])=e(H_2'[Z'_4])\ge |Z'_4|-t\ge \frac{2}{3}|Z'_4|\ge \frac{2}{3}|Z_4|\ge \frac{1}{3}(|Z_4|+|Z_{14}|).
\end{align}
Note that $|Y|+|Z|=n'-|X|-1=n'-3$, 
$|Y|=|Y_1|+|Y_2|+|Y_3|$, $|Z_{11}|\le 6$, and  $$|Z|=|Z_1|+|Z_2|+|Z_3|+|Z_4|=|Z_{11}|+|Z_{12}|+|Z_{13}|+|Z_{14}|+|Z_2|+|Z_{3}|+|Z_{4}|.$$
Combine Inequations (\ref{2eXY}), (\ref{2eYY}), (\ref{eF2}), (\ref{eF}), (\ref{eZ}), we obtain
\begin{align*}
3m'&=3e(H')=e(H_0')=e(R)+e(T)+e(H_2')\\
&\ge m'+n'-1+e(H_2'[X,Y])+e(H_2'[Y])+|F_1|+|F_2|+e(H_2'[Z_4])
\\
&\ge m'+n'-1+\frac{1}{3}(|Y_1|+2|Y_2|+|Z_2|+|Z_{12}|+
|Y_3|+|Z_3|+|Z_{13}|+|Z_{4}|+|Z_{14}|)\\
&\ge m'+n'-1+\frac{1}{3}(|Y|+|Z|-|Z_{11}|)\\
&\ge  m'+n'-1+\frac{1}{3}(n'-9),
\end{align*}
and 
$$e(H')\ge \frac{2n'}{3}-2.$$
Thus 
\begin{align}\label{c4com}
e(H')\geq \begin{cases}
	0, &\text{ when $\delta(H')= 0$;}\\
	\frac{2n'}{3} -2,  &\text{ when $\delta(H')= 1$;}\\
	\frac{13n'-29}{18}, &\text{ when $\delta(H') = 2$;}\\
	n', &\text{ when $\delta(H')\ge 3$.}
\end{cases} 
\end{align}

Observe that the number of components of a Berge-$C_4$-saturated linear $3$-uniform hypergraph is at most 2. Otherwise, we can find a new hyperedge by choosing a vertex from three components, which will not create a Berge-$C_4$, a contradiction.
Assume the two components of $H$ are $H_1$ and $H_2$. By Inequation (\ref{c4com}), 
$e(H_i)\ge \frac{2|V(H_i)|}{3}-2$ for any $i\in [2]$ and so $e(H)\ge \frac{2n}{3}-4.$

$\hfill\qedsymbol$

\end{document}